\documentclass[11pt, a4paper]{article}
\usepackage[T1]{fontenc}
\usepackage{xcolor}
\usepackage{geometry}
\usepackage{amssymb}
\usepackage{mathtools}
\usepackage{amsthm, thmtools}
\usepackage{graphicx,epsfig,subcaption}
\usepackage[inline]{enumitem}
\usepackage{booktabs, multirow}
\usepackage{bm}
\usepackage[ruled, noline, nofillcomment, linesnumbered]{algorithm2e}
\usepackage{hyperref, cleveref}
\usepackage{setspace}
\usepackage{rotating}
\usepackage{float}
\usepackage[square, comma, sort&compress, numbers]{natbib}
\usepackage{easyReview}
\graphicspath{{../figs}{./}}

\numberwithin{equation}{section}

\DontPrintSemicolon
\SetAlgoCaptionSeparator{.}
\SetKwComment{tcc}{$\triangleright$\ }{}

\SetCommentSty{mycommfont}
\SetNlSty{textrm}{}{:}
\hypersetup{
    hidelinks,
}

\declaretheorem{theorem}

\declaretheorem[sibling=theorem]{remark}

\declaretheorem[sibling=theorem]{proposition}

\DeclarePairedDelimiter{\norm}{\lVert}{\rVert}

\DeclarePairedDelimiter{\paren}{\lparen}{\rparen}

\newcommand{\mbf}[1]{\mathbf{#1}}

\newcommand{\bem}{\begin{bmatrix}}\newcommand{\eem}{\end{bmatrix}}

\DeclareMathOperator*{\diag}{\mathrm{diag}}
\DeclareMathOperator*{\range}{\mathrm{range}}
\DeclareMathOperator*{\rank}{\mathrm{rank}}
\DeclareMathOperator*{\nul}{\mathrm{null}}

\title{TriCG with deflated restarting for symmetric quasi-definite linear systems}

\author{Kui Du\thanks{School of Mathematical Sciences, Xiamen University, Xiamen
    361005, China (kuidu@xmu.edu.cn).}, \quad Jia-Jun Fan\thanks{School of
        Mathematical Sciences, Xiamen University, Xiamen 361005, China
    (jiajunfan@stu.xmu.edu.cn).}}

\date{}

\begin{document}

\maketitle

\begin{abstract}
    TriCG is a short-recurrence iterative method recently introduced by Montoison
    and Orban [SIAM J.\ Sci.\ Comput., 43 (2021), pp.\ A2502--A2525] for solving
    symmetric quasi-definite (SQD) linear systems. TriCG takes advantage of the
    inherent block structure of SQD linear systems and performs substantially
    better than SYMMLQ. 
    However, numerical experiments have revealed that the convergence of TriCG can be notably slow when the off-diagonal block contains a substantial number of large elliptic singular values.
    To address this limitation, we introduce a deflation strategy tailored for TriCG to improve its convergence behavior. Specifically, we develop a generalized Saunders--Simon--Yip process with deflated restarting to construct the deflation subspaces. Building upon this process, we propose a novel method termed TriCG with deflated restarting. The deflation subspaces can also be utilized to solve SQD linear systems with multiple right-hand sides. Numerical experiments are provided to illustrate the superior performance of the proposed methods.


\end{abstract}

\section{Introduction}
% \label{sec:introduction}

We consider linear systems of the form
\begin{equation} \label{eq:problem}
    \begin{bmatrix}
        \mathbf{M} & \mathbf{A} \\ 
        \mathbf{A}^\top & -\mathbf{N}
    \end{bmatrix}
    \begin{bmatrix}
        \mathbf{x} \\ 
        \mathbf{y}
    \end{bmatrix} =
    \begin{bmatrix}
        \mathbf{b} \\ \mathbf{c}
    \end{bmatrix}, 
\end{equation}
where $\mathbf{M} \in \mathbb{R}^{m \times m}$ and $\mathbf{N} \in \mathbb{R}^{n
\times n}$ are symmetric positive definite (SPD), $\mathbf{b} \in \mathbb{R}^m$
and $\mathbf{c} \in \mathbb{R}^n$ are nonzero, and $\mathbf{A} \in \mathbb{R}^{m
\times n}$ is an arbitrary nonzero matrix. The coefficient matrix of
\eqref{eq:problem} is called symmetric quasi-definite (SQD)
\cite{orban2017iterative}. SQD linear systems arise in a variety of
applications, for example, computational fluid dynamics
\cite{elman2002preconditioners,elman2014finite}, and optimization problems
\cite{friedlander2012primal}.

SQD matrices are symmetric, indefinite, and nonsingular. Krylov subspace
methods, such as MINRES and SYMMLQ \cite{paige1975solutions}, can be employed to
solve \eqref{eq:problem}. It should be noted that these methods solve the system
as a whole and often exploit the block structure in the preconditioning stage.

Recently, several iterative methods that are specifically
tailored to exploit the block structure of \eqref{eq:problem} have been developed. Based on the
generalized Saunders--Simon--Yip (gSSY) tridiagonalization process
\cite{saunders1988two, buttari2019tridiagonalization}, Montoison and Orban
\cite{montoison2021tricg} proposed two short-recurrence methods called TriCG and
TriMR for solving \eqref{eq:problem}. TriCG and TriMR are mathematically
equivalent to preconditioned Block-CG and Block-MINRES with two right-hand sides, 
in which the two approximate solutions are summed at each iteration. But the
storage and work per iteration of TriCG and TriMR are similar to those of CG
\cite{hestenes1952methods} and MINRES \cite{paige1975solutions}, respectively.
Numerical experiments in \cite{montoison2021tricg} show that TriCG and TriMR
appear to preserve orthogonality in the basis vectors better than preconditioned
Block-CG and Block-MINRES, and terminate earlier than SYMMLQ and MINRES. Du,
Fan, and Zhang \cite{du2025improved} recently proposed improved versions of
TriCG and TriMR that avoid unlucky terminations. They also demonstrated that
the maximum number of iterations at which the gSSY tridiagonalization process terminates is
determined by the rank of $\mathbf{A}$ and the number of distinct elliptic
singular values of $\mathbf{A}$. In addition to iterative methods specifically tailored for SQD linear systems, there are also specially designed iterative methods that exploit the block structure of saddle-point linear systems or block two-by-two nonsymmetric linear systems; see, for example,
\cite{buttari2019tridiagonalization,estrin2018spmr,montoison2023gpmr,orban2017iterative,benzi2005numer,rozloznik2018saddl}.
 
When solving linear systems, deflation refers to mitigating the influence of
specific eigenvalues that tend to slow down the convergence of iterative
methods. Deflation can be implemented by augmenting a subspace with approximate eigenvectors, or by constructing a preconditioner based on eigenvectors. 
Deflation techniques integrated with CG-type methods have been widely developed.
For example, Saad et al.\ \cite{saad2000deflated} proposed a deflated version of
CG by adding some vectors into the Krylov subspace of CG. Dumitrasc, Kruse, and
R\"ude \cite{dumitrasc2024deflation} developed a deflation strategy by deflating
the off-diagonal block in symmetric saddle point systems and applied it with
Craig's method \cite{craig1955n}. For more developments related to deflation we
refer the reader to
\cite{abdelrehim2010deflated,gutknecht2012spectral,soodhalter2020survey,daas2021recycling,du2025deflated}
and the references therein.

Numerical experiments demonstrate that TriCG often exhibits slow convergence
when $\mbf A$ in \eqref{eq:problem} has a substantial number of large elliptic singular values.
To reduce the influence of large elliptic singular values, we can deflate
\eqref{eq:problem} by using corresponding elliptic singular vectors. We show
that the deflated system can still be solved by TriCG. Since the desired
elliptic singular vectors are usually not available in practice, we develop a
gSSY process with deflated restarting to compute their approximations. Combining
this process with TriCG, we propose a new method called TriCG with deflated
restarting (TriCG-DR). The TriCG-DR method is closely related to the methods in
\cite{morgan2002gmres,baglama2005augmented,abdelrehim2010deflated,baglama2013augmented,dumitrasc2024deflation,du2025deflated}.
We also explore solving SQD linear systems with multiple right-hand sides. When TriCG-DR is applied to the system with the first right-hand side, the elliptic singular vector information obtained can be used to improve the convergence of systems with other right-hand sides. We propose a method called deflated TriCG (D-TriCG) to implement this approach effectively. Specifically, we solve the system with the first right-hand side using TriCG-DR, then project subsequent systems using the obtained approximate elliptic singular vectors before applying TriCG. 

This paper is organized as follows. In the remainder of this section, we
introduce some notation. In \cref{sec:ssy}, we review the gSSY
tridiagonalization process and TriCG. In \cref{sec:deflation}, we introduce the
deflated system and present its connection to \eqref{eq:problem}. In
\cref{sec:gssy_dr}, the gSSY process with deflated restarting for
computing several desired elliptic singular values and vectors is proposed. In
\cref{sec:tricgdr}, we introduce TriCG-DR and present its detailed
implementations. In \cref{sec:mrhs}, we introduce D-TriCG for solving SQD linear systems with multiple right-hand sides. Numerical
experiments and concluding remarks are given in
\cref{sec:experiment,sec:conclusion}, respectively.

\emph{Notation}. We use uppercase bold letters to denote matrices, and lowercase
bold letters to denote column vectors unless otherwise specified. We use
$\mathbf{I}_k$ to denote the identity of size $k \times k$. The zero vector or
matrix is denoted by $\mathbf{0}$. The vector $\mathbf{e}_k$ denotes the $k$th
column of the identity matrix $\mathbf{I}$ whose size is clear from the context.
For a vector $\mathbf{v}$, $\mathbf{v}^\top$ and $\|\mathbf{v}\|$ denote its
transpose and 2-norm, respectively. For an SPD matrix $\mathbf{M}$, the unique
SPD square root matrix of $\mbf M$ is denoted by $\mbf M^{\frac{1}{2}}$, and the
$\mathbf{M}$-norm of a vector $\mathbf{v}$ is defined as $\norm{\mathbf{v}}_{\mathbf{M}} =
\sqrt{\mathbf{v}^\top \mathbf{M} \mathbf{v}}$. For a matrix $\mathbf{A}$, its
transpose, inverse, range, and null space are denoted by $\mathbf{A}^\top$,
$\mathbf{A}^{-1}$, $\range(\mathbf{A})$, and $\nul(\mathbf{A})$, respectively.
The normalization of the form ``$\beta \mathbf{Mu} = \mathbf{b}$'' is short for
``$\widetilde{\mathbf{u}} = \mathbf{M}^{-1} \mathbf{b}$; $\beta =
\sqrt{\widetilde{\mathbf{u}}^\top \mathbf{b}}$; if $\beta=0$, then stop, else
$\mathbf{u} = \widetilde{\mathbf{u}} / \beta$.''

\section{The gSSY tridiagonalization process and TriCG} 
\label{sec:ssy}

We first review the gSSY tridiagonalization process. For a general matrix
$\mathbf{A} \in \mathbb{R}^{m \times n}$, SPD matrices $\mathbf{M} \in
\mathbb{R}^{m \times m}$ and $\mathbf{N} \in \mathbb{R}^{n \times n}$, and
nonzero initial vectors $\mathbf{b} \in \mathbb{R}^m$ and $\mathbf{c} \in
\mathbb{R}^n$, we describe the gSSY tridiagonalization process in
\Cref{alg:orth_tri}. 

\begin{algorithm}[htbp]
    \caption{Generalized Saunders--Simon--Yip tridiagonalization process}
    \label{alg:orth_tri}

    \KwIn{SPD matrices $\mathbf{M} \in \mathbb{R}^{m\times m}$ and $\mathbf{N}
        \in \mathbb{R}^{n\times n}$, a general matrix $\mathbf{A} \in
        \mathbb{R}^{m\times n}$, nonzero vectors $\mathbf{b}\in\mathbb{R}^m$ and
        $\mathbf{c} \in \mathbb{R}^n$
    }

    $\mathbf{u}_0 = \mathbf{0}$, $\mathbf{v}_0 = \mathbf{0}$\; $\beta_1
    \mathbf{M} \mathbf{u}_1 = \mathbf{b}$, $\gamma_1 \mathbf{N} \mathbf{v}_1 =
    \mathbf{c}$\;

    \For{$j = 1, 2, \dots$}
    {
        $\mathbf{q} = \mathbf{A} \mathbf{v}_j - \gamma_j \mathbf{M}
        \mathbf{u}_{j-1}$\; 

        $\mathbf{p} = \mathbf{A}^\top \mathbf{u}_j - \beta_j \mathbf{N}
        \mathbf{v}_{j-1}$\; 
        $\alpha_j = \mathbf{u}_j^\top \mathbf{q}$\; 

        $\beta_{j+1} \mathbf{M} \mathbf{u}_{j+1} = \mathbf{q} - \alpha_j
        \mathbf{M} \mathbf{u}_j$ \label{alg:update_u} \;

        $\gamma_{j+1} \mathbf{N} \mathbf{v}_{j+1} =
        \mathbf{p} - \alpha_j \mathbf{N} \mathbf{v}_j$ \label{alg:update_v} \;
    }
\end{algorithm}

After $j$ iterations of \Cref{alg:orth_tri}, the following relations hold:
\begin{subequations} \label{relation_of_orth_tri}
    \begin{align}
        \mathbf{A} \mathbf{V}_j &= \mathbf{M} \mathbf{U}_j \mathbf{T}_j +
        \beta_{j+1} \mathbf{M} \mathbf{u}_{j+1} \mathbf{e}_{j}^\top =
        \mathbf{M} \mathbf{U}_{j+1} \mathbf{T}_{j+1, j}, \\
        \mathbf{A}^\top \mathbf{U}_j &= \mathbf{N} \mathbf{V}_j
        \mathbf{T}_j^\top + \gamma_{j+1} \mathbf{N} \mathbf{v}_{j+1}
        \mathbf{e}_{j}^\top = \mathbf{N} \mathbf{V}_{j+1} \mathbf{T}_{j,j+1}^\top, \\
        \mathbf{U}_j^\top \mathbf{M} \mathbf{U}_j &= \mathbf{V}_j^\top
        \mathbf{N} \mathbf{V}_j = \mathbf{I}_j,\quad 
        \mathbf{U}_j^\top \mathbf{AV}_j = \mathbf{T}_j,
    \end{align} 
\end{subequations}
where  \begin{equation*}
    \mathbf{V}_j =
    \begin{bmatrix}
        \mathbf{v}_1 & \mathbf{v}_2 & \cdots & \mathbf{v}_j
    \end{bmatrix}, \quad
    \mathbf{U}_j =
    \begin{bmatrix}
        \mathbf{u}_1 & \mathbf{u}_2 & \cdots & \mathbf{u}_j
    \end{bmatrix}, 
\end{equation*} and
\begin{equation*}
    \mathbf{T}_j = 
    \begin{bmatrix}
        \alpha_1 & \gamma_2 \\
        \beta_2  & \alpha_2 & \ddots \\
                 & \ddots   & \ddots & \gamma_j \\
                 &			& \beta_j & \alpha_j
    \end{bmatrix}, \quad
    \mathbf{T}_{j,j+1} = 
    \begin{bmatrix}
        \mathbf{T}_j & \gamma_{j+1}\mathbf{e}_j
    \end{bmatrix}, \quad
    \mathbf{T}_{j+1,j} = 
    \begin{bmatrix}
        \mathbf{T}_j \\ \beta_{j+1}\mathbf{e}_j^\top
    \end{bmatrix}.
\end{equation*}

We next review TriCG proposed by Montoison and Orban \cite{montoison2021tricg}.
Utilizing the relations in \eqref{relation_of_orth_tri}, we have
\begin{equation} \label{relation_block}
    \begin{bmatrix}
        \mathbf{M} & \mathbf{A} \\ \mathbf{A}^\top & -\mathbf{N}
    \end{bmatrix}
    \begin{bmatrix}
        \mathbf{U}_j & \\ & \mathbf{V}_j
    \end{bmatrix} = 
    \begin{bmatrix}
        \mathbf{M} & \\ & \mathbf{N}
    \end{bmatrix}
    \begin{bmatrix}
        \mathbf{U}_{j+1} & \\ & \mathbf{V}_{j+1}
    \end{bmatrix}
    \begin{bmatrix}
        \mathbf{I}_{j+1,j} & \mathbf{T}_{j+1,j} \\ \mathbf{T}_{j,j+1}^\top &
        -\mathbf{I}_{j+1, j}
    \end{bmatrix}, 
\end{equation}
where $\mathbf{I}_{j+1,j}$ is the matrix consisting of the first $j$ columns of
$\mathbf{I}_{j+1}$. Let
\begin{equation}\label{relation_block1}
    \mathbf{K} := 
    \begin{bmatrix}
        \mathbf{M} & \mathbf{A} \\
        \mathbf{A}^\top & -\mathbf{N}
    \end{bmatrix}, \quad
    \mathbf{H} :=
    \begin{bmatrix}
        \mathbf{M} & \\
                   & \mathbf{N}
    \end{bmatrix},
\end{equation}
and
\[
    \mathbf{P}_j :=
    \begin{bmatrix}
        \mathbf{e}_1 & \mathbf{e}_{j+1} & \cdots & \mathbf{e}_i &
        \mathbf{e}_{j+i} & \cdots & \mathbf{e}_{j} & \mathbf{e}_{2j} 
    \end{bmatrix} \in \mathbb{R}^{2j \times 2j}
\]
be the permutation matrix introduced by Paige \cite{paige1974bidiagonalization}.
Let
\begin{equation} \label{definition_W}
    \mathbf{W}_j :=
    \begin{bmatrix}
        \mathbf{U}_j & \\ & \mathbf{V}_j
    \end{bmatrix} \mathbf{P}_j. 
\end{equation}
Combining \eqref{relation_block}, \eqref{relation_block1},  and \eqref{definition_W} yields
\[
    \mathbf{K} \mathbf{W}_j = \mathbf{H} \mathbf{W}_{j+1} \mathbf{P}_{j+1}^\top
    \begin{bmatrix}
        \mathbf{I}_{j+1,j} & \mathbf{T}_{j+1,j} \\ 
        \mathbf{T}_{j,j+1}^\top & -\mathbf{I}_{j+1, j}
    \end{bmatrix}\mathbf{P}_j =: \mathbf{H} \mathbf{W}_{j+1} \mathbf{S}_{j+1,j},
\]
where
\[
    \mathbf{S}_{j+1,j} =
    \begin{bmatrix}
        \bm{\Omega}_1 & \bm{\Psi}_2 \\
        \bm{\Psi}_2^\top & \bm{\Omega}_2 & \ddots \\
                         & \ddots     & \ddots & \bm{\Psi}_j \\
                         &			 & \ddots & \bm{\Omega}_j \\
                         &            &        & \bm{\Psi}_{j+1}^\top
    \end{bmatrix}\in\mathbb{R}^{(2j+2) \times 2j}, \quad
    \bm{\Omega}_j = 
    \begin{bmatrix}
        1 & \alpha_j \\ \alpha_j & -1
    \end{bmatrix}, \quad
    \bm{\Psi}_j = 
    \begin{bmatrix}
        0 & \gamma_j \\ \beta_j & 0
    \end{bmatrix}.
\]
Let $\mathbf{S}_j$ denote the leading $2j \times 2j$ submatrix of
$\mathbf{S}_{j+1,j}$. At step $j$, TriCG solves the subproblem
\[
    \mathbf{S}_j \mathbf{z}_j = \beta_1 \mathbf{e}_1 + \gamma_1 \mathbf{e}_2, \quad
    \mathbf{z}_j := 
    \begin{bmatrix}
        \xi_1 & \xi_2 & \cdots & \xi_{2j}
    \end{bmatrix}^\top \in \mathbb{R}^{2j},
\]
and generates the $j$th iterate
\[
    \begin{bmatrix}
        \mathbf{x}_j \\ \mathbf{y}_j
    \end{bmatrix} =
    \mathbf{W}_j \mathbf{z}_j,
\]
which satisfies the Galerkin condition 
\[
    \mathbf{W}_j^\top \mathbf{r}_j = \mathbf{W}_j^\top  
    \biggl(
        \begin{bmatrix} \mathbf{b} \\ \mathbf{c}
            \end{bmatrix} - \begin{bmatrix} \mathbf{M} & \mathbf{A} \\
            \mathbf{A}^\top & -\mathbf{N} \end{bmatrix} \begin{bmatrix} \mathbf{x}_j \\
        \mathbf{y}_j \end{bmatrix}
    \biggr) = \mathbf{0}.
\]
The corresponding residual is (see \cite[(3.13)]{montoison2021tricg})
\begin{equation} \label{eq:rj_tricg}
    \mathbf{r}_j = -\mathbf{H}
    \begin{bmatrix}
        \mathbf{u}_j+1 & \mathbf{0} \\
        \mathbf{0} & \mathbf{v}_{j+1}
    \end{bmatrix}
    \begin{bmatrix}
        \beta_{j+1} \xi_{2j} \\
        \gamma_{j+1} \xi_{2j-1}.
    \end{bmatrix}
\end{equation}

The $\mathrm{LDL}^\top$ factorization $\mathbf{S}_j = \mathbf{L}_j \mathbf{D}_j
\mathbf{L}_j^\top$ with 
\[
    \mathbf{D}_j = \begin{bmatrix} 
                       d_1 & & \\ & \ddots & \\ & & d_{2 j}
                   \end{bmatrix}, \quad
    \mathbf{L}_j = \begin{bmatrix} 
                       \bm{\Delta}_1 & & & \\
                       \bm{\Gamma}_2 & \bm{\Delta}_2 & & \\ 
                                     & \ddots & \ddots & \\ 
                                     & & \bm{\Gamma}_j & \bm{\Delta}_j 
                   \end{bmatrix}, \quad
    \bm{\Delta}_j = \begin{bmatrix} 
                        1 & \\
                        \delta_j & 1 
                    \end{bmatrix}, \quad
    \bm{\Gamma}_j = \begin{bmatrix} 
                        & \sigma_j \\ 
                        \eta_j & \lambda_j 
                    \end{bmatrix}
\]
exists, and can be obtained via the following recurrences
\begin{subequations} \label{eq:ldl}
    \begin{alignat}{2}
        d_{2j-1} & = 1-\sigma_j^2 d_{2j-2}, &\qquad & j \geq 1, \label{eq:ldl1}\\
        \delta_j & =\left(\alpha_j-\lambda_j \beta_j \right) / d_{2j-1}, & & j \geq 1, \\        d_{2j} & = -1-\eta_j^2 d_{2j-3}-\lambda_j^2 d_{2j-2}-\delta_j^2 d_{2j-1}, & & j \geq 1, \\
        \sigma_j & =\beta_j / d_{2j-2}, & & j \geq 2, \\
        \eta_j & =\gamma_j / d_{2j-3}, & & j \geq 2, \\
        \lambda_j & =-\gamma_j \delta_{j-1} / d_{2j-2}, & & j \geq 2,
        \label{eq:ldl2}
    \end{alignat}
\end{subequations} with $d_{-1}=d_0=\sigma_1=\eta_1=\lambda_1=0$.
By utilizing the $\mathrm{LDL}^\top$ factorization and the strategy of Paige and
Saunders \cite{paige1975solutions}, Montoison and Orban
\cite{montoison2021tricg} showed that the $k$th iterate of TriCG can be updated
via short recurrences. For the convenience of the subsequent discussion, we
present the iterative scheme here. Let 
\[
    \mathbf{p}_j = \mathbf{D}_j^{-1} \mathbf{L}_j^{-1} (\beta_1 \mathbf{e}_1 +
                   \gamma_1 \mathbf{e}_2) =: 
    \begin{bmatrix}
        \pi_1 & \pi_2 & \cdots & \pi_{2j}
    \end{bmatrix}^\top, \quad
    \mathbf{G}_j = \mathbf{W}_j \mathbf{L}_j^{-\top} =: 
    \begin{bmatrix}
        \mathbf{g}_1^x & \mathbf{g}_2^x & \cdots & \mathbf{g}_{2j}^x \\
        \mathbf{g}_1^y & \mathbf{g}_2^y & \cdots & \mathbf{g}_{2j}^y
    \end{bmatrix}.
\]
We have the recurrences
\begin{subequations} \label{eq:pk}
    \begin{align}
        \pi_{2 j-1} 
        &= 
        \begin{cases}
            \beta_1 / d_1, & j=1, \\
            -\beta_j  \pi_{2 j-2} / d_{2 j-1}, & j \geq 2,
        \end{cases} \\
        \pi_{2 j} 
        &= 
        \begin{cases}
            \left(
                \gamma_1-\delta_1 \beta_1
            \right) / d_2, & j=1, \\
            -\left(
                \delta_j d_{2 j-1} \pi_{2 j-1}+\lambda_j d_{2 j-2} \pi_{2
                j-2}+\gamma_j  \pi_{2 j-3}
            \right) / d_{2 j}, & j \geq 2,
        \end{cases}
    \end{align}
\end{subequations}
and
\begin{subequations} \label{eq:g}
    \begin{align}
        \mbf g_{2 j-1}^x & =-\sigma_j \mbf g_{2 j-2}^x+\mbf u_j \label{eq:g1} \\
        \mbf g_{2 j-1}^y & =-\sigma_j \mbf g_{2 j-2}^y \\
        \mbf g_{2 j}^x & =-\delta_j \mbf g_{2 j-1}^x-\lambda_j \mbf g_{2 j-2}^x-\eta_j \mbf g_{2 j-3}^x \\
       \mbf  g_{2 j}^y & =-\delta_j \mbf g_{2  j-1}^y-\lambda_j \mbf g_{2 j-2}^y-\eta_j \mbf g_{2
        j-3}^y+\mbf v_j, \label{eq:g2}
    \end{align}
\end{subequations}
with  $\mathbf{g}_{-1}^x = \mathbf{g}_0^x = \mathbf{0}$ and
$\mathbf{g}_{-1}^y = \mathbf{g}_0^y =\mathbf{0}$.
The $j$th iterate is updated by
\begin{align*}
    \mathbf{x}_j &= \mathbf{x}_{j-1} + \pi_{2j-1}\mathbf{g}_{2j-1}^x +
    \pi_{2j}\mathbf{g}_{2j}^x, \\
    \mathbf{y}_j &= \mathbf{y}_{j-1} + \pi_{2j-1}\mathbf{g}_{2j-1}^y +
    \pi_{2j}\mathbf{g}_{2j}^y,
\end{align*} with  $\mbf x_0=\mbf 0$ and $\mbf y_0=\mbf 0$.
The corresponding residual norm is
\[
    \norm{\mathbf{r}_j}_{\mathbf{H}^{-1}} = 
    \begin{cases}
        \sqrt{\gamma_1^2 + \beta_1^2}, & j = 0, \\
        \sqrt{\bigl(\gamma_{j+1}(\pi_{2j-1}-\delta_j\pi_{2j})\bigr)^2 +
        (\beta_{j+1}\pi_{2j})^2}, & j \ge 1.
    \end{cases}
\]

\section{Deflation of elliptic singular values} 
\label{sec:deflation}

In this section, we will introduce deflation techniques to mitigate the
influence of large elliptic singular values. Given two SPD matrices
$\mathbf{M} \in \mathbb{R}^{m\times m}$ and $\mathbf{N} \in \mathbb{R}^{n\times
n}$, the elliptic singular value decomposition (ESVD)
\cite{arioli2013generalized} of a matrix $\mathbf{A} \in \mathbb{R}^{m\times n}$
is defined as below
\[
    \mathbf{A} = \mathbf{M} \widetilde{\mathbf{U}} \bm{\Sigma}
    \widetilde{\mathbf{V}}^\top \mathbf{N},
\]
where $\widetilde{\mathbf{U}} \in \mathbb{R}^{m \times m}$ and
$\widetilde{\mathbf{V}} \in \mathbb{R}^{n \times n}$ satisfy
$\widetilde{\mathbf{U}}^\top \mathbf{M} \widetilde{\mathbf{U}} = \mathbf{I}_m$
and $\widetilde{\mathbf{V}}^\top \mathbf{N} \widetilde{\mathbf{V}} =
\mathbf{I}_n$, and $\bm{\Sigma} \in \mathbb{R}^{m \times n}$ is a diagonal
matrix whose diagonal elements $\sigma_i$ are nonnegative and in nonincreasing
order (i.e., $\sigma_1 \ge \sigma_2 \ge \cdots \ge \sigma_d \ge 0$, $d =
\min\{m, n\}$). Clearly, the ESVD of $\mathbf{A}$ is equivalent to the standard
SVD of $\mathbf{M}^{-\frac{1}{2}} \mathbf{A} \mathbf{N}^{-\frac{1}{2}}$. 

Now we consider the two-sided preconditioned matrix $\mathbf{H}^{-\frac{1}{2}} \mathbf{K} \mathbf{H}^{-\frac{1}{2}}$. From the ESVD of $\mathbf{A}$, we have
\[
    \mathbf{H}^{-\frac{1}{2}} \mathbf{K} \mathbf{H}^{-\frac{1}{2}} =\bem \mbf I
    & \mathbf{M}^{-\frac{1}{2}} \mathbf{A}
    \mathbf{N}^{-\frac{1}{2}} \\ \mathbf{N}^{-\frac{1}{2}} \mathbf{A}^\top
    \mathbf{M}^{-\frac{1}{2}} & -\mbf I \eem= \bem \mbf
    M^{\frac{1}{2}}\widetilde{\mbf U} & \\ & \mbf N^{\frac{1}{2}}\widetilde{\mbf
    V} \eem \begin{bmatrix}
        \mathbf{I} & \bm{\Sigma} \\
        \bm{\Sigma}^\top & -\mathbf{I}
    \end{bmatrix}
    \bem\widetilde{\mbf U}^\top \mbf M^{\frac{1}{2}} & \\ & \widetilde{\mbf V}^\top\mbf N^{\frac{1}{2}} \eem.
\]
Since $\mathbf{M}^{\frac{1}{2}} \widetilde{\mathbf{U}}$ and $\mathbf{N}^{\frac{1}{2}} \widetilde{\mathbf{V}}$ are both orthogonal, the eigenvalues of the
preconditioned matrix $\mathbf{H}^{-\frac{1}{2}} \mathbf{K}
\mathbf{H}^{-\frac{1}{2}}$ are
\[
    \lambda 
    \bigl(\mathbf{H}^{-\frac{1}{2}} \mathbf{K}
\mathbf{H}^{-\frac{1}{2}}
    \bigr) = \lambda 
    \biggl(
        \begin{bmatrix}
            \mathbf{I} & \bm{\Sigma} \\
            \bm{\Sigma}^\top & -\mathbf{I}
        \end{bmatrix}
    \biggr) = 
    \begin{cases}
        \pm \sqrt{\sigma_i^2 + 1}, & i = 1, \dots, r, \\
        1, & (m - r) \text{ times}, \\
        -1, & (n - r) \text{ times},
    \end{cases}
\]
where $r = \rank(\mathbf{A})$. This suggests that the elliptic singular values
of $\mathbf{A}$ affect the eigenvalue distribution of
$\mathbf{H}^{-\frac{1}{2}} \mathbf{K} \mathbf{H}^{-\frac{1}{2}}$, and the
spectrum of $\mathbf{H}^{-\frac{1}{2}} \mathbf{K} \mathbf{H}^{-\frac{1}{2}}$ is
confined to the interval $[-\sqrt{\sigma_1^2+1}, -1] \cup [1,
\sqrt{\sigma_1^2+1}]$.

Next we introduce a deflation strategy to improve the eigenvalue distribution. 
Let $\widetilde{\mathbf{U}}_k$ and $\widetilde{\mathbf{V}}_k$ be the matrices
consisting of the first $k$ columns of $\widetilde{\mathbf{U}}$ and
$\widetilde{\mathbf{V}}$, and let $\bm{\Sigma}_k$ denote the leading $ k \times k $
submatrix of $\bm{\Sigma}$.  
We
have the following relations:
\[
    \mathbf{A} \widetilde{\mathbf{V}}_k = \mathbf{M} \widetilde{\mathbf{U}}_k
    \bm{\Sigma}_k,\quad \mathbf{A}^\top \widetilde{\mathbf{U}}_k =
    \mathbf{N} \widetilde{\mathbf{V}}_k \bm{\Sigma}_k.
\]
Define two projectors $\mathbf{P}$ and $\mathbf{Q}$ as follows
\[
    \mathbf{P} = \mathbf{I} - \mathbf{M} \widetilde{\mathbf{U}}_k
    \widetilde{\mathbf{U}}_k^\top,\quad \mathbf{Q} = \mathbf{I} -
    \widetilde{\mathbf{V}}_k \widetilde{\mathbf{V}}_k^\top \mathbf{N}.
\]
We have 
\begin{equation} \label{relations_of_projections}
      \mathbf{P}=\mathbf{P}^2,\quad \mathbf{Q}=\mathbf{Q}^2,\quad \mathbf{P} \mathbf{M} =
    \mathbf{M} \mathbf{P}^\top, \quad \mathbf{N} \mathbf{Q} = \mathbf{Q}^\top
    \mathbf{N}, \quad \mathbf{PA} = \mathbf{AQ} = \mathbf{PAQ}. 
\end{equation}
For convenience, let $\mathbf{f} = \begin{bmatrix} \mathbf{b}^\top
& \mathbf{c}^\top \end{bmatrix}^\top$ and 
\begin{equation}\label{orth_projection}
    \bm{\mathcal{P}} =  
    \begin{bmatrix}
        \mathbf{P} & \\
                   & \mathbf{Q}^\top
    \end{bmatrix}.
\end{equation}  
We define the deflated system as
\begin{equation} \label{eq:deflated}
    \bm{\mathcal{P}} \mathbf{K} \widetilde{\mathbf{u}} = \bm{\mathcal{P}} \mathbf{f}. 
\end{equation} 
Straightforward computations yield 
\[
    \lambda 
    \left(\mathbf{H}^{-\frac{1}{2}} \bm{\mathcal{P}}\mathbf{K}
\mathbf{H}^{-\frac{1}{2}}
    \right) = 
    \begin{cases}
        \pm \sqrt{\sigma_i^2 + 1}, & i = k+1, \dots, r, \\
        1, & (m - r) \text{ times}, \\
        -1, & (n - r) \text{ times},\\
        0, & 2k \text{ times}.
    \end{cases}
\]
Since $\rank(\bm{\mathcal{P}}) = m + n - 2k$, applying $\bm{\mathcal{P}}$ does
not preserve the solution set, i.e., \eqref{eq:deflated} and
\eqref{eq:problem} are not equivalent. The following theorem tells us how to obtain the solution from the
deflated system \eqref{eq:deflated}.

\begin{theorem} \label{thm:correct_x}
    Let $\widetilde{\mathbf{u}}$ be a solution of the deflated system
    \eqref{eq:deflated}. Then, the solution of the system \eqref{eq:problem} is
    given by
    \begin{equation} \label{eq:correct_u}
        \mathbf{u} = \mathbf{Z}_k \left( \mathbf{Z}_k^\top \mathbf{KZ}_k
        \right)^{-1} \mathbf{Z}_k^\top \mathbf{f} + \bm{\mathcal{P}}^\top
        \widetilde{\mathbf{u}}, \quad 
        \mathbf{Z}_k = 
        \begin{bmatrix}
            \widetilde{\mathbf{U}}_k & \\
            & \widetilde{\mathbf{V}}_k
        \end{bmatrix}.
    \end{equation}
\end{theorem}
\begin{proof}
    It is
    straightforward to verify that
    \[ \bm{\mathcal{P}} \mathbf{K} = \mathbf{K} \bm{\mathcal{P}}^\top,\quad 
        \bm{\mathcal{P}} = \mathbf{I} - \mathbf{KZ}_k \left( \mathbf{Z}_k^\top
        \mathbf{K} \mathbf{Z}_k \right)^{-1} \mathbf{Z}_k^\top. 
    \]
    Then, we have
    \begin{align*}
        \mathbf{Ku} &= \mathbf{K} \mathbf{Z}_k \left( \mathbf{Z}_k^\top
        \mathbf{KZ}_k \right)^{-1} \mathbf{Z}_k^\top \mathbf{f} + \mathbf{K}
        \bm{\mathcal{P}}^\top \widetilde{\mathbf{u}} \\
        &= \left( \mathbf{I} - \bm{\mathcal{P}} \right) \mathbf{f} +
        \bm{\mathcal{P}} \mathbf{K} \widetilde{\mathbf{u}} = \mathbf{f}.
        \qedhere
    \end{align*} 
\end{proof}
Note that $\mathbf{Z}_k^\top \mathbf{KZ}_k \in \mathbb{R}^{2k \times 2k}$, and
thus the computational cost of $\mathbf{Z}_k \left( \mathbf{Z}_k^\top \mathbf{KZ}_k
\right)^{-1} \mathbf{Z}_k^\top \mathbf{f}$ is not significant. If the deflated system \eqref{eq:deflated} is solved approximately and the approximate solution of \eqref{eq:problem} is obtained by \eqref{eq:correct_u}, the following proposition provides the relations between the residuals and errors of \eqref{eq:problem} and \eqref{eq:deflated}.

\begin{proposition} 
    Let $\mathbf{u}_\star$, $\widetilde{\mathbf{u}}_\star$, and
    $\widetilde{\mathbf{u}}$ be the exact solution of \eqref{eq:problem}, an
    exact solution of \eqref{eq:deflated}, and an approximate solution of \eqref{eq:deflated},
    respectively. If $\mathbf{u}$ is obtained via \eqref{eq:correct_u}, then we
    have
    \[
        \mathbf{f} - \mathbf{Ku} = 
        \bm{\mathcal{P}} (\mathbf{f} - \mathbf{K} \widetilde{\mathbf{u}}),\qquad \mathbf{u}_\star - \mathbf{u} = \bm{\mathcal{P}}^\top
        (\widetilde{\mathbf{u}}_\star - \widetilde{\mathbf{u}}),
    \]
    and
    \[
        \norm{\mathbf{u}_\star - \mathbf{u}}_\mathbf{H} \le
        \norm{\widetilde{\mathbf{u}}_\star - \widetilde{\mathbf{u}}}_\mathbf{H},
    \]
\end{proposition}

\begin{proof}
    From \eqref{eq:correct_u} and $\bm{\mathcal{P}} \mathbf{K} = \mathbf{K}
    \bm{\mathcal{P}}^\top$, we have
    \[
        \mathbf{f} - \mathbf{Ku} = \mathbf{f} - (\mathbf{I} - \bm{\mathcal{P}})
        \mathbf{f} - \bm{\mathcal{P}} \mathbf{K} \widetilde{\mathbf{u}} =
        \bm{\mathcal{P}} (\mathbf{f} - \mathbf{K} \widetilde{\mathbf{u}}).
    \]
    From $\nul(\bm{\mathcal{P}} \mathbf{K}) = \nul(\bm{\mathcal{P}}^\top)$,
    $\widetilde{\mathbf{u}}_\star$ can be represented by
    \[
        \widetilde{\mathbf{u}}_\star = \mathbf{u}_\star + \mathbf{z}, \quad
        \mathbf{z} \in \nul(\bm{\mathcal{P}}^\top).
    \]
    From \Cref{thm:correct_x}, we have
    \[
        \mathbf{u}_\star = \mathbf{Z}_k \bigl( \mathbf{Z}_k^\top \mathbf{KZ}_k
        \bigr)^{-1} \mathbf{Z}_k^\top \mathbf{f} + \bm{\mathcal{P}}^\top
        \widetilde{\mathbf{u}}_\star.
    \]
    Since $\mathbf{u}$ is obtained via \eqref{eq:correct_u}, we have 
    \[
        \mathbf{u}_\star - \mathbf{u} = \bm{\mathcal{P}}^\top
        (\widetilde{\mathbf{u}}_\star - \widetilde{\mathbf{u}}).
    \]
    Since $(\mathbf{I} - \bm{\mathcal{P}}^\top)^\top \mathbf{H}
    \bm{\mathcal{P}}^\top = \mathbf{0}$, for any $\mathbf{y} \in
    \mathbb{R}^{m+n}$, we have
    \[
        \|\mathbf{y}\|_\mathbf{H}^2 = \|\bm{\mathcal{P}}^\top \mathbf{y} +
        (\mathbf{I} - \bm{\mathcal{P}}^\top) \mathbf{y} \|_\mathbf{H}^2 =
        \|\bm{\mathcal{P}}^\top \mathbf{y}\|_\mathbf{H}^2 + \|(\mathbf{I} -
        \bm{\mathcal{P}}^\top) \mathbf{y} \|_\mathbf{H}^2 \ge
        \|\bm{\mathcal{P}}^\top \mathbf{y}\|_\mathbf{H}^2.
    \]
    It follows that
    \[
        \norm{\mathbf{u}_\star - \mathbf{u}}_\mathbf{H} =
        \norm{\bm{\mathcal{P}}^\top (\widetilde{\mathbf{u}}_\star -
        \widetilde{\mathbf{u}})}_\mathbf{H} \le
        \norm{\widetilde{\mathbf{u}}_\star - \widetilde{\mathbf{u}}}_\mathbf{H}.
        \qedhere
    \]
\end{proof}

Now we show that \eqref{eq:deflated} can also be solved by TriCG.
From \eqref{relations_of_projections}, \eqref{eq:deflated} can be rewritten as
\begin{equation} \label{eq:deflated2}
    \begin{bmatrix}
        \mathbf{PM} & \mathbf{AQ} \\
        \mathbf{Q}^\top \mathbf{A}^\top & -\mathbf{Q}^\top\mathbf{N}
    \end{bmatrix}
    \begin{bmatrix}
        \widetilde{\mathbf{x}} \\ \widetilde{\mathbf{y}}
    \end{bmatrix} =
    \begin{bmatrix}
        \mathbf{Pb} \\ \mathbf{Q}^\top \mathbf{c}
    \end{bmatrix}.
\end{equation}
To establish the relationship between the gSSY tridiagonalization process
for the system \eqref{eq:problem} and for the deflated system \eqref{eq:deflated2}, we need the following theorem.

\begin{theorem} \label{thm:deflated_recurrence} 
Assume
    that \Cref{alg:orth_tri} with $\mathbf{b}$ and $\mathbf{c}$ replaced by $\mathbf{Pb}$ and
    $\mathbf{Q}^\top \mathbf{c}$ does not terminate at the first $k$ iterations. Then the generated $\mathbf{U}_{k+1}$ and $\mathbf{V}_{k+1}$ satisfy
    \[
        \range(\mathbf{U}_{k+1}) \subseteq \mathbf{M}^{-1} \range(\mathbf{P}),\quad
        \range(\mathbf{V}_{k+1}) \subseteq \mathbf{N}^{-1} \range(\mathbf{Q}^\top).
    \]
\end{theorem}

\begin{proof}
    The proof is by induction on $k$. Since $\beta_1 \mathbf{M} \mathbf{u}_1 =
    \mathbf{Pb}$ and $\gamma_1 \mathbf{N} \mathbf{v}_1 = \mathbf{Q}^\top
    \mathbf{c}$, we have $\mathbf{u}_1 \in \mathbf{M}^{-1} \range(\mathbf{P})$ and
    $\mathbf{v}_1 \in \mathbf{N}^{-1} \range(\mathbf{Q}^\top)$. We assume that
    the following relations hold:
    \[
        \range(\mathbf{U}_k) \subseteq \mathbf{M}^{-1}
        \range(\mathbf{P}),\quad \range(\mathbf{V}_k) \subseteq \mathbf{N}^{-1} \range(\mathbf{Q^\top}).
    \]
    From \eqref{relations_of_projections}, we obtain
    \[
        \mathbf{QN}^{-1}\mathbf{Q}^\top = \mathbf{N}^{-1}\mathbf{Q}^\top, \quad
        \mathbf{P}^\top\mathbf{M}^{-1}\mathbf{P} = \mathbf{M}^{-1}\mathbf{P}.
    \]
    Since $\mathbf{u}_k \in \mathbf{M}^{-1} \range(\mathbf{P})$, we have $\mathbf{u}_k =
    \mathbf{M}^{-1} \mathbf{Pz} = \mathbf{P}^\top\mathbf{M}^{-1}\mathbf{Pz} =
    \mathbf{P}^\top \mathbf{u}_k$ for some $\mathbf{z} \in \mathbb{R}^m$.
    Similarly, $\mathbf{v}_k = \mathbf{Qv}_k$ holds. Thus, $\mathbf{Av}_k =
    \mathbf{AQv}_k = \mathbf{PAv}_k$ and $\mathbf{A}^\top
    \mathbf{u}_k = \mathbf{A}^\top \mathbf{P}^\top \mathbf{u}_k =
    \mathbf{Q}^\top \mathbf{A}^\top \mathbf{u}_k$. From lines
    \ref{alg:update_u}--\ref{alg:update_v} of \Cref{alg:orth_tri}, we have
    \begin{align*}
        \beta_{k+1}\mathbf{Mu}_{k+1} &= \mathbf{Av}_k - \gamma_k
        \mathbf{Mu}_{k-1} - \alpha_k \mathbf{Mu}_k \\
        &= \mathbf{PAv}_k - \gamma_k \mathbf{Mu}_{k-1} - \alpha_k \mathbf{Mu}_k
        \in \range(\mathbf{P}),
        \shortintertext{and}
        \gamma_{k+1}\mathbf{Nv}_{k+1} &= \mathbf{A}^\top \mathbf{u}_k - \beta_k
        \mathbf{Nv}_{k-1} - \alpha_k \mathbf{Nv}_k \\
        &= \mathbf{Q}^\top \mathbf{A}^\top \mathbf{u}_k - \beta_k
        \mathbf{Nv}_{k-1} - \alpha_k \mathbf{Nv}_k \in \range(\mathbf{Q}^\top).
    \end{align*}
    Therefore,
    \[
        \mathbf{u}_{k+1} \in \mathbf{M}^{-1} \range(\mathbf{P}),\quad
        \mathbf{v}_{k+1} \in \mathbf{N}^{-1} \range(\mathbf{Q}^\top). \qedhere
    \] \end{proof}
 
Let $\mathbf{U}_k$ and
$\mathbf{V}_k$ be the matrices generated by \Cref{alg:orth_tri} with the input
$\{\mathbf{M}, \mathbf{N}, \mathbf{A}, \mathbf{Pb}, \mathbf{Q}^\top\mathbf{c}\}$. 
Then by \eqref{relation_of_orth_tri}, \eqref{relations_of_projections}, and
\Cref{thm:deflated_recurrence}, we have
\begin{align*}
    \begin{bmatrix}
        \mathbf{PM} & \mathbf{AQ} \\
        \mathbf{Q}^\top \mathbf{A}^\top & -\mathbf{Q}^\top\mathbf{N}
    \end{bmatrix}
    \begin{bmatrix}
        \mathbf{U}_k \\ & \mathbf{V}_k
    \end{bmatrix} 
    &= 
    \begin{bmatrix}
        \mathbf{PMU}_{k+1} \\ & \mathbf{Q}^\top \mathbf{NV}_{k+1}
    \end{bmatrix}
    \begin{bmatrix}
        \mathbf{I}_{k+1,k} & \mathbf{T}_{k+1,k} \\
        \mathbf{T}_{k,k+1}^\top & -\mathbf{I}_{k+1,k}
    \end{bmatrix} \\
    &=
    \begin{bmatrix}
        \mathbf{MU}_{k+1} \\ & \mathbf{NV}_{k+1}
    \end{bmatrix}
    \begin{bmatrix}
        \mathbf{I}_{k+1,k} & \mathbf{T}_{k+1,k} \\
        \mathbf{T}_{k,k+1}^\top & -\mathbf{I}_{k+1,k}
    \end{bmatrix}.
\end{align*}
This observation suggests that the
deflated system \eqref{eq:deflated2} can also be solved by utilizing TriCG in the same manner as the system \eqref{eq:problem}, and
the only modification is to replace $\mathbf{b}$ and $\mathbf{c}$ with
$\mathbf{Pb}$ and $\mathbf{Q}^\top\mathbf{c}$. 

The above results are obtained from the exact partial ESVD of $\mathbf{A}$ with $k$ elliptic singular triplets. However, in
practice, the exact partial ESVD of a matrix usually is not readily available. 
Fortunately, satisfactory numerical results can be obtained by using approximate
elliptic singular triplets. We mention that Dumitrasc et al.\ \cite{dumitrasc2024deflation}
proposed two iterative algorithm for computing approximate elliptic singular
triplets. In the next section, we introduce a gSSY process with deflated
restarting for computing approximate elliptic singular triplets, which can be
used in TriCG for deflation.

\section{A gSSY process with deflated restarting}
\label{sec:gssy_dr}

In this section, we introduce an algorithm to compute approximate elliptic singular triplets. The proposed  algorithm is closely related to the Lanczos-DR algorithm in \cite{abdelrehim2010deflated}. Recall that an elliptic singular triplet $\{\sigma_j,
\mathbf{u}_j, \mathbf{v}_j\}$ of $\mathbf{A}$ satisfies
\[
    \mathbf{Av}_j = \sigma_j \mathbf{M} \mathbf{u}_j, \quad \mathbf{A}^\top
    \mathbf{u}_j = \sigma_j \mathbf{N} \mathbf{v}_j.
\]
After $p$ iterations of \Cref{alg:orth_tri}, the following relations hold
\begin{equation} \label{SSY_relations2}
    \begin{split}
        \mathbf{AV}_p &= \mathbf{M} \mathbf{U}_p \mathbf{T}_p + \beta_{p+1}
        \mathbf{M} \mathbf{u}_{p+1} \mathbf{e}_p^\top, \\
        \mathbf{A}^\top \mathbf{U}_p &= \mathbf{N} \mathbf{V}_p
        \mathbf{T}_p^\top + \gamma_{p+1} \mathbf{N} \mathbf{v}_{p+1}
        \mathbf{e}_p^\top.
    \end{split} 
\end{equation}
Consider the SVD
\[
    \mathbf{T}_p = \widehat{\mathbf{U}} \widehat{\bm{\Sigma}}
    \widehat{\mathbf{V}}^\top
\]
where $\widehat{\mathbf{U}} = \begin{bmatrix} \widehat{\mathbf{u}}_1 &
\widehat{\mathbf{u}}_2 & \cdots & \widehat{\mathbf{u}}_p \end{bmatrix}$ and
$\widehat{\mathbf{V}} = \begin{bmatrix} \widehat{\mathbf{v}}_1 &
\widehat{\mathbf{v}}_2 & \cdots & \widehat{\mathbf{v}}_p \end{bmatrix}$ are orthogonal, and $\widehat{\bm{\Sigma}}$ is
diagonal with diagonal elements in nonincreasing order: $\widehat{\sigma}_1
\ge \widehat{\sigma}_2 \ge \dotsb \ge \widehat{\sigma}_p \ge 0$. 
Let\begin{equation} \label{approx_singular_triplets}
    \widetilde{\sigma}_j := \widehat{\sigma}_j, \quad
    \widetilde{\mathbf{u}}_j := \mathbf{U}_p \widehat{\mathbf{u}}_j, \quad  
    \widetilde{\mathbf{v}}_j := \mathbf{V}_p \widehat{\mathbf{v}}_j. 
\end{equation}
Combining \eqref{SSY_relations2} and \eqref{approx_singular_triplets} yields
\begin{equation} \label{relations_approx_singular_triplets}
    \begin{split}
        \mathbf{A} \widetilde{\mathbf{v}}_j &= \widetilde{\sigma}_j \mathbf{M}
        \widetilde{\mathbf{u}}_j + \beta_{p+1} \mathbf{M} \mathbf{u}_{p+1} (
        \mathbf{e}_p^\top \widehat{\mathbf{v}}_j ), \\
        \mathbf{A}^\top \widetilde{\mathbf{u}}_j &= \widetilde{\sigma}_j
        \mathbf{N} \widetilde{\mathbf{v}}_j + \gamma_{p+1} \mathbf{N}
        \mathbf{v}_{p+1} ( \mathbf{e}_p^\top \widehat{\mathbf{u}}_j ).
    \end{split} 
\end{equation}
The relations in \eqref{relations_approx_singular_triplets} suggest that the triplet $\{\widetilde{\sigma}_j,
\widetilde{\mathbf{u}}_j, \widetilde{\mathbf{v}}_j\}$ can be accepted as an
approximate elliptic singular triplet of $\mbf A$ if $\beta_{p+1} |\mathbf{e}_p^\top
\widehat{\mathbf{v}}_j|$ and $\gamma_{p+1} |\mathbf{e}_p^\top
\widetilde{\mathbf{u}}_j|$ are sufficiently small. In our algorithm we accepts $\{\widetilde{\sigma}_j, \widetilde{\mathbf{u}}_j,
\widetilde{\mathbf{v}}_j\}$ as an approximate elliptic singular triplet of $\mathbf{A}$ if 
\begin{equation} \label{eq:svd_stop_criteria}
    \max \big\{ \beta_{p+1} |\mathbf{e}_p^\top \widehat{\mathbf{v}}_j|,
    \gamma_{p+1} |\mathbf{e}_p^\top \widehat{\mathbf{u}}_j | \big\} \le
    \varepsilon_\mathrm{svd}.
\end{equation}

Assume that our objective is to compute the $k$ largest elliptic singular
triplets of $\mathbf{A}$. (Other elliptic singular
triplets can be computed similarly.)
Let $\widetilde{\mathbf{u}}_j$ and
$\widetilde{\mathbf{v}}_j$ for $1 \le j \le k$ be the vectors in
\eqref{approx_singular_triplets}. If \eqref{eq:svd_stop_criteria} does not hold
for some $j$, we improve the triplets in a deflated restarting fashion. The
strategy used here is closely related to that in \cite{baglama2013augmented}.
More precisely, we define
\begin{equation} \label{eq:new_basis_gssy-dr}
    \widetilde{\mathbf{V}}_k := \mathbf{V}_p \widehat{\mathbf{V}}_k, \quad
    \widetilde{\mathbf{U}}_k := \mathbf{U}_p \widehat{\mathbf{U}}_k, \quad
    \widetilde{\mathbf{V}}_{k+1} := 
    \begin{bmatrix} 
        \widetilde{\mathbf{V}}_k & \mathbf{v}_{p+1} 
    \end{bmatrix}, \quad 
    \widetilde{\mathbf{U}}_{k+1} := 
    \begin{bmatrix}
        \widetilde{\mathbf{U}}_k & \mathbf{u}_{p+1} 
    \end{bmatrix},
\end{equation}
where $\widehat{\mathbf{V}}_k$ and $\widehat{\mathbf{U}}_k$ are the matrices
consisting of the first $k$ columns of $\widehat{\mathbf{V}}$ and
$\widehat{\mathbf{U}}$, respectively. It follows from
\eqref{relations_approx_singular_triplets} that
\begin{align}
    \mathbf{A} \widetilde{\mathbf{V}}_{k+1} &= 
    \begin{bmatrix} 
        \mathbf{M} \mathbf{U}_p \widehat{\mathbf{U}}_k \widehat{\bm{\Sigma}}_k +
        \beta_{p+1} \mathbf{M} \mathbf{u}_{p+1} (\mathbf{e}_p^\top
        \widehat{\mathbf{V}}_k) & \mathbf{A} \mathbf{v}_{p+1} 
    \end{bmatrix}, \label{A_tilde_V}
    \shortintertext{and}
    \mathbf{A}^\top \widetilde{\mathbf{U}}_{k+1} &= 
    \begin{bmatrix}
        \mathbf{N} \mathbf{V}_p \widehat{\mathbf{V}}_k \widehat{\bm{\Sigma}}_k +
        \gamma_{p+1} \mathbf{N} \mathbf{v}_{p+1} (\mathbf{e}_p^\top
        \widehat{\mathbf{U}}_k) & \mathbf{A}^\top \mathbf{u}_{p+1} 
    \end{bmatrix}, \label{A_trans_tilde_U}
\end{align} where $\widehat{\bm{\Sigma}}_k=\diag\{\widehat{\sigma}_1,\widehat{\sigma}_2,\ldots,\widehat{\sigma}_k\}$.
Let $\widetilde{\mathbf{u}}_{k+2}$ and $\widetilde{\mathbf{v}}_{k+2}$ be defined as
\begin{equation} \label{eq:uk+2}
    \widetilde{\beta}_{k+2} \mathbf{M} \widetilde{\mathbf{u}}_{k+2} :=
    \left( \mathbf{I} - \mathbf{M} \widetilde{\mathbf{U}}_{k+1}
    \widetilde{\mathbf{U}}_{k+1}^\top \right) \mathbf{Av}_{p+1},\quad
    \widetilde{\gamma}_{k+2} \mathbf{N} \widetilde{\mathbf{v}}_{k+2} :=
    \left( \mathbf{I} - \mathbf{N} \widetilde{\mathbf{V}}_{k+1}
    \widetilde{\mathbf{V}}_{k+1}^\top \right) \mathbf{A}^\top \mathbf{u}_{p+1},
\end{equation}
respectively. Here, $\widetilde{\beta}_{k+2}$ and $\widetilde{\gamma}_{k+2}$ are
scaling factors such that $\|\widetilde{\mathbf{u}}_{k+2}\|_{\mathbf{M}} =
\|\widetilde{\mathbf{v}}_{k+2}\|_{\mathbf{N}} = 1$. Using the relations in
\eqref{relations_approx_singular_triplets} and the orthogonality, we obtain
\[
    \widetilde{\mathbf{U}}_{k}^\top \mathbf{Av}_{p+1} = 
    \bigl(\mathbf{N}
        \widetilde{\mathbf{V}}_k \widehat{\bm{\Sigma}}_k + \gamma_{p+1}
        \mathbf{N}\mathbf{v}_{p+1} \mathbf{e}_p^\top \widehat{\mathbf{U}}_k 
    \bigr)^\top 
    \mathbf{v}_{p+1} = \gamma_{p+1} \widehat{\mathbf{U}}_k^\top \mathbf{e}_p.
\]
Therefore, we have
\begin{equation} \label{orth_onto_U}
    \widetilde{\beta}_{k+2} \mathbf{M} \widetilde{\mathbf{u}}_{k+2} =
    \mathbf{Av}_{p+1} - \mathbf{M} \widetilde{\mathbf{U}}_k \bigl( \gamma_{p+1}
    \widehat{\mathbf{U}}_k^\top \mathbf{e}_p \bigr) - \widetilde{\alpha}_{k+1}
    \mathbf{M} \mathbf{u}_{p+1},\qquad \widetilde{\alpha}_{k+1} :=
    \mathbf{u}_{p+1}^\top \mathbf{Av}_{p+1}. 
\end{equation}
Similarly, we have
\begin{align}
    \widetilde{\gamma}_{k+2} \mathbf{N} \widetilde{\mathbf{v}}_{k+2} =
    \mathbf{A}^\top \mathbf{u}_{p+1} - \mathbf{N} \widetilde{\mathbf{V}}_k
    \bigl( \beta_{p+1} \widehat{\mathbf{V}}_k^\top \mathbf{e}_p \bigr) -
    \widetilde{\alpha}_{k+1} \mathbf{N} \mathbf{v}_{p+1}. \label{orth_onto_V}
\end{align}
Thus, substituting \eqref{orth_onto_U} and \eqref{orth_onto_V} into
\eqref{A_tilde_V} and \eqref{A_trans_tilde_U}, respectively, yields 
\begin{equation} 
    \begin{aligned}
        \mathbf{A} \widetilde{\mathbf{V}}_{k+1} &= \mathbf{M}
        \widetilde{\mathbf{U}}_{k+1} \widetilde{\mathbf{T}}_{k+1} +
        \widetilde{\beta}_{k+2} \mathbf{M} \widetilde{\mathbf{u}}_{k+2}
        \mathbf{e}_{k+1}^\top, \\
        \mathbf{A}^\top \widetilde{\mathbf{U}}_{k+1} &= \mathbf{N}
        \widetilde{\mathbf{V}}_{k+1} \widetilde{\mathbf{T}}_{k+1}^\top +
        \widetilde{\gamma}_{k+2} \mathbf{N} \widetilde{\mathbf{v}}_{k+2}
        \mathbf{e}_{k+1}^\top,
    \end{aligned} \label{deflated_decomposition}
\end{equation}
where $\widetilde{\mathbf{T}}_{k+1}$ is an arrow-shaped matrix of the form
\begin{align*}
    \widetilde{\mathbf{T}}_{k+1} =
    \begin{bmatrix}
        \widehat{\bm{\Sigma}}_k & \gamma_{p+1} \widehat{\mathbf{U}}_k^\top \mathbf{e}_p \\
        \beta_{p+1} \mathbf{e}_p^\top \widehat{\mathbf{V}}_k & \widetilde{\alpha}_{k+1}
    \end{bmatrix} =:
    \begin{bmatrix}
        \widetilde{\alpha}_1 & & & \widetilde{\gamma}_2 \\
        & \ddots & & \vdots \\
        &        & \ddots & \widetilde{\gamma}_{k+1} \\
        \widetilde{\beta}_2 & \dots & \widetilde{\beta}_{k+1} & \widetilde{\alpha}_{k+1}
    \end{bmatrix}.
\end{align*}
Note that $\widetilde{\alpha}_j = \widehat{\sigma}_j$ for $j \le k$. We continue 
generating the basis vectors in a similar fashion to \eqref{eq:uk+2} by
\begin{align*}
    \widetilde{\beta}_{j+1} \mathbf{M} \widetilde{\mathbf{u}}_{j+1} &:= 
    \bigl( 
        \mathbf{I} - \mathbf{M} \widetilde{\mathbf{U}}_{j}
        \widetilde{\mathbf{U}}_{j}^\top 
    \bigr) \mathbf{A} \widetilde{\mathbf{v}}_{j} 
    \shortintertext{and}
    \widetilde{\gamma}_{j+1} \mathbf{N} \widetilde{\mathbf{v}}_{j+1} &:= 
    \bigl( 
        \mathbf{I} - \mathbf{N} \widetilde{\mathbf{V}}_{j}
        \widetilde{\mathbf{V}}_{j}^\top 
    \bigr) \mathbf{A}^\top \widetilde{\mathbf{u}}_{j}
\end{align*}
for $j = k+2, k+3, \dots, p$. Utilizing \eqref{deflated_decomposition}, we
obtain
\begin{align*}
    \widetilde{\beta}_{k+3} \mathbf{M} \widetilde{\mathbf{u}}_{k+3} 
    &= \mathbf{A} \widetilde{\mathbf{v}}_{k+2} - \mathbf{M} 
    \widetilde{\mathbf{U}}_{k+1} 
    \bigl( 
        \mathbf{A}^\top
        \widetilde{\mathbf{U}}_{k+1} 
    \bigr)^\top \widetilde{\mathbf{v}}_{k+2} -
    \bigl( 
        \widetilde{\mathbf{u}}_{k+2}^\top \mathbf{A}
        \widetilde{\mathbf{v}}_{k+2} 
    \bigr) \mathbf{M} \widetilde{\mathbf{u}}_{k+2} \\
    &= \mathbf{A} \widetilde{\mathbf{v}}_{k+2} - \mathbf{M}
    \widetilde{\mathbf{U}}_{k+1} 
    \bigl( \widetilde{\mathbf{V}}_{k+1}
        \widetilde{\mathbf{T}}_{k+1}^\top + \widetilde{\gamma}_{k+2}
        \widetilde{\mathbf{v}}_{k+2} \mathbf{e}_{k+1}^\top 
    \bigr)^\top \mathbf{N}
    \widetilde{\mathbf{v}}_{k+2} - \widetilde{\alpha}_{k+2} \mathbf{M}
    \widetilde{\mathbf{u}}_{k+2} \\
    &= \mathbf{A} \widetilde{\mathbf{v}}_{k+2} - \widetilde{\gamma}_{k+2}
    \mathbf{M} \widetilde{\mathbf{u}}_{k+1} - \widetilde{\alpha}_{k+2} \mathbf{M}
    \widetilde{\mathbf{u}}_{k+2}.
\end{align*}
Similarly, we have
\[
    \widetilde{\gamma}_{k+3} \mathbf{N} \widetilde{\mathbf{v}}_{k+3} =
    \mathbf{A}^\top \widetilde{\mathbf{u}}_{k+2} - \widetilde{\beta}_{k+2}
    \mathbf{N} \widetilde{\mathbf{v}}_{k+1} - \widetilde{\alpha}_{k+2} \mathbf{N}
    \widetilde{\mathbf{v}}_{k+2}.
\]
This means that the basis vectors
$\widetilde{\mathbf{u}}_{j+1}$ and $\widetilde{\mathbf{u}}_{j+1}$ for $j = k+3, \dots, p$, can be obtained via the same three-term recurrences as those of \Cref{alg:orth_tri}. 
And we have the relations
\begin{equation} \label{new_deflated_decomposition}
    \begin{split}
        \mathbf{A} \widetilde{\mathbf{V}}_{p} &= \mathbf{M}
        \widetilde{\mathbf{U}}_{p} \widetilde{\mathbf{T}}_{p} +
        \widetilde{\beta}_{p+1} \mathbf{M} \widetilde{\mathbf{u}}_{p+1}
        \mathbf{e}_{p}^\top, \\
        \mathbf{A}^\top \widetilde{\mathbf{U}}_{p} 
        &= \mathbf{N} \widetilde{\mathbf{V}}_{p} \widetilde{\mathbf{T}}_{p}^\top
        + \widetilde{\gamma}_{p+1} \mathbf{N} \widetilde{\mathbf{v}}_{p+1}
        \mathbf{e}_{p}^\top, \\
        \widetilde{\mathbf{V}}_p \mathbf{N} \widetilde{\mathbf{V}}_p &=
        \widetilde{\mathbf{V}}_p \mathbf{N} \widetilde{\mathbf{V}}_p =
        \mathbf{I}_p,\quad \widetilde{\mathbf{T}}_p =
        \widetilde{\mathbf{U}}_p^\top \mathbf{A} \widetilde{\mathbf{V}}_p,
    \end{split}
\end{equation}
which are analogous to \eqref{SSY_relations2}, but
\[
    \widetilde{\mathbf{T}}_p = 
    \begin{bmatrix}
        \widetilde{\alpha}_1 & & & \widetilde{\gamma}_2 \\
                             & \ddots & & \vdots \\
                             &        & \ddots & \widetilde{\gamma}_{k+1} \\
        \widetilde{\beta}_2 & \dots & \widetilde{\beta}_{k+1} &
        \widetilde{\alpha}_{k+1} & \widetilde{\gamma}_{k+2} \\
                                 & & & \widetilde{\beta}_{k+2} &
        \widetilde{\alpha}_{k+2} & \ddots \\
                                 & & & & \ddots & \ddots &
                                 \widetilde{\gamma}_{p} \\
                                 & & & & & \widetilde{\beta}_{p} &
                                 \widetilde{\alpha}_{p}
    \end{bmatrix}
\]
is no longer tridiagonal. 
Replacing \eqref{SSY_relations2} with
\eqref{new_deflated_decomposition} and repeating the above procedure yields a new algorithm called the gSSY process with deflated restarting (gSSY-DR($p,k$)) for computing approximate partial ESVD of $\mathbf{A}$. We present the implementation of gSSY-DR($p,k$) in \Cref{alg:approx_svd}.

\begin{algorithm}[htbp]
    \caption{gSSY-DR$(p,k)$}
    \label{alg:approx_svd}

    \KwIn{$\mathbf{M}$, $\mathbf{N}$, $\mathbf{A}$, $\mathbf{b}$, $\mathbf{c}$,
        $p$--maximum subspace dimension, 
        $k$--number of desired elliptic singular triplets,
        $\varepsilon_\mathrm{svd}$--tolerance for approximate elliptic singular triplets, 
        \texttt{maxcycle}--maximum number of cycles.
    }

    \KwOut{Approximate left and right elliptic singular vectors $\mathbf{U}_k$ and
        $\mathbf{V}_k$, 
        and approximate elliptic singular values $\bm{\Sigma}_k$.
    }

    $\beta_1 \mathbf{M} \mathbf{u}_1 = \mathbf{b}$, $\gamma_1 \mathbf{N}
    \mathbf{v}_1 = \mathbf{c}$ \;

    $\mathbf{U}_1 = [\mathbf{u}_1]$, $\mathbf{V}_1 = [\mathbf{v}_1]$\;
    $k_{\rm aug} = k$, $k = 0$ \tcc*[r]{Set the dimension of augmentation to zero
    before the first cycle}

    \For(\tcc*[f]{Outer cycle}){$\mathtt{outerit} = 1,2,\dots,\mathtt{maxcycle}$}
    {
        $\mathbf{q} = \mathbf{Av}_{k+1} - \mathbf{MU}_k \mathbf{T}_{1:k,k+1}$
        \label{alg_line:gssy_outer_q}\;

        $\mathbf{p} = \mathbf{A}^\top \mathbf{u}_{k+1} - \mathbf{NV}_k
        \mathbf{T}_{k+1,1:k}^\top$\;

        $\alpha_{k+1} = \mathbf{u}_{k+1}^\top \mathbf{q}$\;

        $\beta_{k+2} \mathbf{Mu}_{k+2} = \mathbf{q} - \alpha_{k+1}
        \mathbf{u}_{k+1}$\;

        $\gamma_{k+2} \mathbf{Nv}_{k+2} = \mathbf{p} - \alpha_{k+1}
        \mathbf{v}_{k+1}$ \label{alg_line:gssy_outer_v}\;

        $\mathbf{T}_{k+1,k+1} = \alpha_{k+1}$\;

        \For {$j = k+2, k+3, \dots, p$}
        {
            $\mathbf{q} = \mathbf{Av}_j - \gamma_j \mathbf{M} \mathbf{u}_{j-1}$
            \label{alg_line:gssy_inner_q}\;

            $\mathbf{p} = \mathbf{A}^\top \mathbf{u}_j - \beta_j \mathbf{N}
            \mathbf{v}_{j-1}$\;

            $\alpha_j = \mathbf{u}_j^\top \mathbf{q}$\;

            $\mathbf{M} \mathbf{u} = \mathbf{q} - \alpha_j \mathbf{M}
            \mathbf{u}_j$\;
            
            $\mathbf{N} \mathbf{v} = \mathbf{p} - \alpha_j \mathbf{N}
            \mathbf{v}_j$ \label{alg_line:gssy_inner_v}\;
            
            $\mathbf{T}_{j, j} = \alpha_j$, $\mathbf{T}_{j-1, j} = \gamma_{j}$,
            $\mathbf{T}_{j, j-1} = \beta_{j}$ \label{alg_line:gssy_update_T} \;

            $\mathbf{U}_{j} = \begin{bmatrix} \mathbf{U}_{j-1} & \mathbf{u}_j
            \end{bmatrix}$, $\mathbf{V}_{j} = \begin{bmatrix} \mathbf{V}_{j-1} 
            & \mathbf{v}_j \end{bmatrix}$\;

            Reorthogonalization: $\mathbf{M} \mathbf{u} = (\mathbf{I} -
            \mathbf{M} \mathbf{U}_j \mathbf{U}_j^\top) \mathbf{M} \mathbf{u}$
            \label{alg_line:reorth_u}\; 

            Reorthogonalization: $\mathbf{N} \mathbf{v} = (\mathbf{I} -
            \mathbf{N} \mathbf{V}_j \mathbf{V}_j^\top) \mathbf{N} \mathbf{v}$
            \label{alg_line:reorth_v} \label{alg_line:gssy_reorth_v}\; 

            $\beta_{j+1} = \|\mathbf{u}\|_{\mathbf{M}}$, $\gamma_{j+1} =
            \|\mathbf{v}\|_{\mathbf{N}}$\;

            $\mathbf{u}_{j+1} = \mathbf{u} / \beta_{j+1}$, $\mathbf{v}_{j+1} =
            \mathbf{v} / \gamma_{j+1}$\;
        }

        $k = k_{\rm aug}$ \tcc*[r]{Recover dimension of augmentation to $k$} Compute
        the SVD of $\mathbf{T}$, and store the $k$ desired elliptic singular
        triplets in $\widehat{\mathbf{U}}_k,\ \bm{\Sigma}_k$, and 
        $\widehat{\mathbf{V}}_k$ \;

        Let $\mathbf{U}_k = \mathbf{U}_p \widehat{\mathbf{U}}_k$ and
        $\mathbf{V}_k = \mathbf{V}_p \widehat{\mathbf{V}}_k$ \;

        \For(\tcc*[f]{Check the number of converged elliptic singular triplets}) 
        {$i = 1,2,\dots,k$ \label{line:gssy_check_converged_num1}} 
        {
            $\mathtt{num\_conv\_sv} = 0$ \;
            \If{
                $\max \{ \beta_{p+1} |\mathbf{e}_p^\top \widehat{\mathbf{V}}_k
                \mathbf{e}_i|, \gamma_{p+1} |\mathbf{e}_p^\top
                \widehat{\mathbf{U}}_k\mathbf{e}_i| \} \le
                \varepsilon_\mathrm{svd}$
            }
            {
                $\mathtt{num\_conv\_sv} = \mathtt{num\_conv\_sv} + 1$
            }
        } \label{line:gssy_check_converged_num2}

        \lIf{$\mathtt{num\_conv\_sv} = k$}{\textbf{stop}}

        $\mathbf{U}_{k+1} = \begin{bmatrix} \mathbf{U}_k & \mathbf{u}_{p+1}
        \end{bmatrix}$, $\mathbf{V}_{k+1} = \begin{bmatrix} \mathbf{V}_k &
        \mathbf{v}_{p+1} \end{bmatrix}$\; \smallskip

        $\mathbf{T}_{1:k,1:k} = \bm{\Sigma}_k$, $\mathbf{T}_{1:k,k+1}
        = \gamma_{p+1} \widehat{\mathbf{U}}_k^\top \mathbf{e}_p$,
        $\mathbf{T}_{k+1,1:k} = \beta_{p+1} \mathbf{e}_p^\top
        \widehat{\mathbf{V}}_k$\;
    }
\end{algorithm}

\begin{remark}
    The reorthogonalization steps (lines
    \ref{alg_line:reorth_u} and \ref{alg_line:reorth_v}) in
    \Cref{alg:approx_svd} are used to control the rounding errors. More reorthogonalization strategies can be employed, such as partial
    and selective reorthogonalization \cite{parlett1998symmetric}. 
\end{remark}

At the end of this section, we analyze the impact of employing the approximate elliptic
singular triplets computed from gSSY-DR($p,k$) (\Cref{alg:approx_svd}) on the deflation strategy proposed in the previous section.
We now use the approximate elliptic
singular vectors to construct the deflated system \eqref{eq:deflated} and compute the approximate solution of \eqref{eq:problem} via \eqref{eq:correct_u}.
The following theorem shows that
the upper bound of the residual norm of \eqref{eq:problem} is dictated by both
the residual norm of \eqref{eq:deflated} and the accuracy level of the
approximate elliptic singular triplets. If the solution of \eqref{eq:deflated}
and the approximate elliptic singular triplets are sufficiently accurate, we can
obtain a good enough solution for \eqref{eq:problem}.

\begin{theorem}
    Assume that the approximate elliptic singular triplets $\{ \widetilde{\mathbf{U}}_k,
    \widetilde{\mathbf{V}}_k, \widetilde{\bm{\Sigma}}_k \}$ are obtained via {\rm
    gSSY-DR($p,k$) (\Cref{alg:approx_svd})} with the stopping criterion
    \eqref{eq:svd_stop_criteria}. Assume that the projections $\mathbf{P}$, $\mathbf{Q}$,
    and $\bm{\mathcal{P}}$ in \eqref{orth_projection} are constructed using $\{
    \widetilde{\mathbf{U}}_k, \widetilde{\mathbf{V}}_k \}$. Let
    $\widetilde{\mathbf{u}}$ be an approximate solution of $\bm{\mathcal{P}}
    \mathbf{Ku} = \bm{\mathcal{P}} \mathbf{f}$ and let $\mathbf{u}$ be computed via
    \eqref{eq:correct_u}. Then, for the residual
    norm, it holds that
    \[
        \norm{ \mathbf{f} - \mathbf{Ku} }_{\mathbf{H}^{-1}} \le \norm{
        \bm{\mathcal{P}} (\mathbf{f} - \mathbf{K} \widetilde{\mathbf{u}})
        }_{\mathbf{H}^{-1}} + \varepsilon_\mathrm{svd} \sqrt{k} \paren[\big]{
        (1+\widetilde{\sigma}_k^2)^{-\frac{1}{2}}
        \norm{\mathbf{f}}_{\mathbf{H}^{-1}} + \sqrt{2} \norm{\mathbf{H}
        \widetilde{\mathbf{u}}}_{\mathbf{H}^{-1}} }.
    \]
\end{theorem}

\begin{proof}
    From \eqref{relations_approx_singular_triplets}, we have the following relations
    \begin{alignat*}{2}
        \mathbf{A} \widetilde{\mathbf{V}}_k 
        &= \mathbf{M}
        \widetilde{\mathbf{U}}_k \widetilde{\bm{\Sigma}}_k + \mathbf{E}_u,
        &\quad \mathbf{E}_u &= \beta_{p+1}
        \mathbf{Mu}_{p+1} \mathbf{e}_p^\top
        \widehat{\mathbf{V}}_k, \\
        \mathbf{A}^\top \widetilde{\mathbf{U}}_k 
        &= \mathbf{N}
        \widetilde{\mathbf{V}}_k \widetilde{\bm{\Sigma}}_k + \mathbf{E}_v,
        &\quad \mathbf{E}_v &= \gamma_{p+1} \mathbf{Nv}_{p+1} \mathbf{e}_p^\top
        \widehat{\mathbf{U}}_k.
    \end{alignat*}
    Since $\mathbf{M}^{\frac{1}{2}} \mathbf{u}_{p+1}$ and
    $\mathbf{N}^{\frac{1}{2}} \mathbf{v}_{p+1}$ are orthonormal, by
    \eqref{eq:svd_stop_criteria}, we have
    \[
        \max \{ 
            \norm{\mathbf{M}^{-\frac{1}{2}} \mathbf{E}_u},
            \norm{\mathbf{N}^{-\frac{1}{2}} \mathbf{E}_v} 
        \} = 
        \max \{
            \beta_{p+1} \|\mathbf{e}_p^\top \widehat{\mathbf{V}}_k\|, 
            \gamma_{p+1} \|\mathbf{e}_p^\top \widehat{\mathbf{U}}_k\| 
        \} \le \varepsilon_\mathrm{svd} \sqrt{k}.
    \]
    We now present the relation between $\mathbf{PA}$ and
    $\mathbf{AQ}$. It follows that
    \begin{align*}
        \mathbf{PA} 
        &= \mathbf{A} - \mathbf{M} \widetilde{\mathbf{U}}_k
        (\mathbf{A}^\top \widetilde{\mathbf{U}}_k)^\top = \mathbf{A} -
        \mathbf{M} \widetilde{\mathbf{U}}_k (\mathbf{N} \widetilde{\mathbf{V}}_k
        \widetilde{\bm{\Sigma}}_k + \mathbf{E}_v)^\top \\
        &= \mathbf{A} - (\mathbf{M} \widetilde{\mathbf{U}}_k
        \widetilde{\bm{\Sigma}}_k) \widetilde{\mathbf{V}}_k^\top \mathbf{N} -
        \mathbf{M} \widetilde{\mathbf{U}}_k \mathbf{E}_v^\top \\
        &= \mathbf{A} - (\mathbf{A} \widetilde{\mathbf{V}}_k - \mathbf{E}_u)
        \widetilde{\mathbf{V}}_k^\top \mathbf{N} - \mathbf{M}
        \widetilde{\mathbf{U}}_k \mathbf{E}_v^\top \\
        &= \mathbf{A} (\mathbf{I} - \widetilde{\mathbf{V}}_k
        \widetilde{\mathbf{V}}_k^\top \mathbf{N}) + \mathbf{E}_u
        \widetilde{\mathbf{V}}_k^\top \mathbf{N} - \mathbf{M}
        \widetilde{\mathbf{U}}_k \mathbf{E}_v^\top \\
        &=\mathbf{AQ} +\mathbf{E}_u
        \widetilde{\mathbf{V}}_k^\top \mathbf{N} - \mathbf{M}
        \widetilde{\mathbf{U}}_k \mathbf{E}_v^\top.
    \end{align*} Define $\mathbf{E}_P:=\mathbf{E}_u
        \widetilde{\mathbf{V}}_k^\top \mathbf{N} - \mathbf{M}
        \widetilde{\mathbf{U}}_k \mathbf{E}_v^\top.$
    Since $(\mathbf{M}^{-\frac{1}{2}}\mathbf{E}_u)^\top \mathbf{M}^{\frac{1}{2}}\widetilde{\mathbf{U}}_k = \mathbf{0}$, $\mathbf{N}^{\frac{1}{2}} \widetilde{\mathbf{V}}_k$ and
    $\mathbf{M}^{\frac{1}{2}} \widetilde{\mathbf{U}}_k$ have orthonormal columns,
     we have 
    \begin{align*}
        \norm{ \mathbf{M}^{-\frac{1}{2}} \mathbf{E}_P \mathbf{N}^{-\frac{1}{2}}} & = \|\mathbf{M}^{-\frac{1}{2}} \mathbf{E}_u
        \widetilde{\mathbf{V}}_k^\top \mathbf{N}^{\frac{1}{2}}-\mathbf{M}^{\frac{1}{2}} \widetilde{\mathbf{U}}_k \mathbf{E}_v^\top
        \mathbf{N}^{-\frac{1}{2}}\| \\
        &\le
        \paren[\big]{ \norm{\mathbf{M}^{-\frac{1}{2}} \mathbf{E}_u
        \widetilde{\mathbf{V}}_k^\top \mathbf{N}^{\frac{1}{2}}}^2 +
        \norm{\mathbf{M}^{\frac{1}{2}} \widetilde{\mathbf{U}}_k \mathbf{E}_v^\top
        \mathbf{N}^{-\frac{1}{2}}}^2 }^{\frac{1}{2}} \\
        &= 
        \bigl(
            \|\mathbf{M}^{-\frac{1}{2}} \mathbf{E}_u\|^2 +
            \|\mathbf{E}_v^\top \mathbf{N}^{-\frac{1}{2}}\|^2
        \bigr)^{\frac{1}{2}}
        \\&\le 
        \varepsilon_\mathrm{svd} \sqrt{2k}.
    \end{align*}
Moreover, we have
\begin{align*}
    \mathbf{KZ}_k 
    &= 
    \begin{bmatrix}
        \mathbf{M} \widetilde{\mathbf{U}}_k & \mathbf{A}
        \widetilde{\mathbf{V}}_k \\
        \mathbf{A}^\top \widetilde{\mathbf{U}}_k & -\mathbf{N}
        \widetilde{\mathbf{V}}_k
    \end{bmatrix} =
    \begin{bmatrix}
        \mathbf{M} \widetilde{\mathbf{U}}_k & \mathbf{M}
        \widetilde{\mathbf{U}}_k \widetilde{\bm{\Sigma}}_k \\
        \mathbf{N} \widetilde{\mathbf{V}}_k \widetilde{\bm{\Sigma}}_k &
        -\mathbf{N} \widetilde{\mathbf{V}}_k
    \end{bmatrix} +
    \begin{bmatrix}
        \mathbf{0} & \mathbf{E}_u \\
        \mathbf{E}_v & \mathbf{0}
    \end{bmatrix} \\
    &= 
    \begin{bmatrix}
        \mathbf{M} \widetilde{\mathbf{U}}_k \\
            & \mathbf{N} \widetilde{\mathbf{V}}_k
    \end{bmatrix}
    \begin{bmatrix}
        \mathbf{I}_k & \widetilde{\bm{\Sigma}}_k \\
        \widetilde{\bm{\Sigma}}_k & -\mathbf{I}_k
    \end{bmatrix} + 
    \begin{bmatrix}
        \mathbf{0} & \mathbf{E}_u \\
        \mathbf{E}_v & \mathbf{0}
    \end{bmatrix},
\end{align*}
and
\begin{align*}
    \bm{\mathcal{P}} \mathbf{K} 
    &= 
    \begin{bmatrix}
        \mathbf{PM} & \mathbf{PA} \\
        \mathbf{Q}^\top \mathbf{A}^\top & -\mathbf{Q}^\top \mathbf{N} 
    \end{bmatrix}
    =
    \begin{bmatrix}
        \mathbf{MP}^\top & \mathbf{PA} \\
        \mathbf{Q}^\top \mathbf{A}^\top & -\mathbf{NQ} 
    \end{bmatrix} = 
    \begin{bmatrix}
        \mathbf{MP}^\top & \mathbf{AQ} + \mathbf{E}_P \\
        \mathbf{A}^\top \mathbf{P}^\top - \mathbf{E}_P^\top & -\mathbf{NQ} 
    \end{bmatrix} \\
    &= 
    \begin{bmatrix}
        \mathbf{MP}^\top & \mathbf{AQ} \\
        \mathbf{A}^\top \mathbf{P}^\top & -\mathbf{NQ}
    \end{bmatrix} +
    \begin{bmatrix}
        \mathbf{0} & \mathbf{E}_P \\
        -\mathbf{E}_P^\top & \mathbf{0}
    \end{bmatrix}.
\end{align*}
Define $$\bm{\mathcal{L}}_k:=\begin{bmatrix}
        \mathbf{I}_k & \widetilde{\bm{\Sigma}}_k \\
        \widetilde{\bm{\Sigma}}_k & -\mathbf{I}_k
    \end{bmatrix},\quad\bm{\mathcal{E}}_Z:=\begin{bmatrix}
        \mathbf{0} & \mathbf{E}_u \\
        \mathbf{E}_v & \mathbf{0}
    \end{bmatrix},\quad  \bm{\mathcal{E}}_P:=\begin{bmatrix}
        \mathbf{0} & \mathbf{E}_P \\
        -\mathbf{E}_P^\top & \mathbf{0}
    \end{bmatrix}.$$ 
We have
\[
    \mathbf{H}^{-\frac{1}{2}} \bm{\mathcal{E}}_Z = 
    \begin{bmatrix}
        \mathbf{0} & \mathbf{M}^{-\frac{1}{2}} \mathbf{E}_u \\
        \mathbf{N}^{-\frac{1}{2}} \mathbf{E}_v & \mathbf{0}
    \end{bmatrix}, \quad
    \mathbf{H}^{-\frac{1}{2}} \bm{\mathcal{E}}_\mathcal{P}
    \mathbf{H}^{-\frac{1}{2}} = 
    \begin{bmatrix}
        \mathbf{0} & \mathbf{M}^{-\frac{1}{2}} \mathbf{E}_P
        \mathbf{N}^{-\frac{1}{2}} \\
        -\mathbf{N}^{-\frac{1}{2}} \mathbf{E}_P^\top \mathbf{M}^{-\frac{1}{2}} &
        \mathbf{0}
    \end{bmatrix}. 
\]
It follows that
\[
    \norm{\mathbf{H}^{-\frac{1}{2}} \bm{\mathcal{E}}_Z} \le \max \{
    \norm{\mathbf{M}^{-\frac{1}{2}} \mathbf{E}_u},
    \norm{\mathbf{N}^{-\frac{1}{2}} \mathbf{E}_v} \} \le \varepsilon_\mathrm{svd}
    \sqrt{k}
\]
and
\[
    \norm{ \mathbf{H}^{-\frac{1}{2}} \bm{\mathcal{E}}_\mathcal{P}
    \mathbf{H}^{-\frac{1}{2}} } = \norm{ \mathbf{M}^{-\frac{1}{2}} \mathbf{E}_P
    \mathbf{N}^{-\frac{1}{2}} } \le \varepsilon_\mathrm{svd} \sqrt{2k}.
\]
Using $ \mathbf{KZ}_k=\mbf H\mbf Z_k\bm{\mathcal{L}}_k+\bm{\mathcal{E}}_Z$, $\mbf Z_k^\top\bm{\mathcal{E}}_Z=\mbf 0$, and $\bm{\mathcal{P}} \mathbf{K}= \mathbf{K}\bm{\mathcal{P}}^\top+\bm{\mathcal{E}}_P$, we obtain
\begin{align}
    \mathbf{Ku} 
    &= \mathbf{K} \mathbf{Z}_k \left( \mathbf{Z}_k^\top \mathbf{KZ}_k \right)^{-1} \mathbf{Z}_k^\top \mathbf{f} + \mathbf{K} \bm{\mathcal{P}}^\top \widetilde{\mathbf{u}} \notag \\
    &= (\mathbf{HZ}_k \bm{\mathcal{L}}_k + \bm{\mathcal{E}}_Z) \bm{\mathcal{L}}_k^{-1}
    \mathbf{Z}_k^\top \mathbf{f} + (\bm{\mathcal{P}} \mathbf{K} -
    \bm{\mathcal{E}}_\mathcal{P}) \widetilde{\mathbf{u}} \notag \\
    &= (\mathbf{I} - \bm{\mathcal{P}}) \mathbf{f} + \bm{\mathcal{E}}_Z
    \bm{\mathcal{L}}_k^{-1} \mathbf{Z}_k^\top \mathbf{f} + (\bm{\mathcal{P}}
    \mathbf{K} - \bm{\mathcal{E}}_\mathcal{P}) \widetilde{\mathbf{u}} \notag \\
    &= \mathbf{f} - \bm{\mathcal{P}} (\mathbf{f} - \mathbf{K}
    \widetilde{\mathbf{u}}) + \bm{\mathcal{E}}_Z \bm{\mathcal{L}}_k^{-1}
    \mathbf{Z}_k^\top \mathbf{f} - \bm{\mathcal{E}}_\mathcal{P} \widetilde{\mathbf{u}}. \label{eq:ku}
\end{align}
Note that we have 
\begin{align}
    \norm{ \bm{\mathcal{E}}_Z \bm{\mathcal{L}}_k^{-1} \mathbf{Z}_k^\top
    \mathbf{f}}_{\mathbf{H}^{-1}} 
    &= \norm{ \mathbf{H}^{-\frac{1}{2}} \bm{\mathcal{E}}_Z
    \bm{\mathcal{L}}_k^{-1} \mathbf{Z}_k^\top \mathbf{H}^{\frac{1}{2}}
    \mathbf{H}^{-\frac{1}{2}} \mathbf{f}} \notag\\
    &\le \|\mathbf{H}^{-\frac{1}{2}} \bm{\mathcal{E}}_Z\| \cdot
    \|\bm{\mathcal{L}}_k^{-1}\| \cdot \|\mathbf{Z}_k^\top
    \mathbf{H}^{\frac{1}{2}}\| \cdot \|\mathbf{H}^{-\frac{1}{2}} \mathbf{f}\| \notag\\
    &\le \varepsilon_\mathrm{svd} \sqrt{k}
    (1+\widetilde{\sigma}^2_k)^{-\frac{1}{2}}
    \norm{\mathbf{f}}_{\mathbf{H}^{-1}},
\end{align}
and
\begin{equation} \label{eq:bound}
    \norm{ \bm{\mathcal{E}}_\mathcal{P}
    \widetilde{\mathbf{u}}}_{\mathbf{H}^{-1}} = \norm{ \mathbf{H}^{-\frac{1}{2}}
    \bm{\mathcal{E}}_\mathcal{P} \mathbf{H}^{-\frac{1}{2}} \mathbf{H}^{\frac{1}{2}}
    \widetilde{\mathbf{u}} } \le \|\mathbf{H}^{-\frac{1}{2}}
    \bm{\mathcal{E}}_\mathcal{P} \mathbf{H}^{-\frac{1}{2}}\| \cdot
    \|\mathbf{H}^{\frac{1}{2}} \widetilde{\mathbf{u}}\| \le
    \varepsilon_\mathrm{svd} \sqrt{2k} \norm{ \mathbf{H} \widetilde{\mathbf{u}}}_{\mathbf{H}^{-1}}. 
\end{equation}
Combining \eqref{eq:ku}--\eqref{eq:bound} yields
\begin{align*}
        \norm{ \mathbf{f} - \mathbf{Ku} }_{\mathbf{H}^{-1}} &= \|\bm{\mathcal{P}} (\mathbf{f} - \mathbf{K}
    \widetilde{\mathbf{u}}) - \bm{\mathcal{E}}_Z \bm{\mathcal{L}}_k^{-1}
   \mathbf{Z}_k^\top \mathbf{f} + \bm{\mathcal{E}}_\mathcal{P} \widetilde{\mathbf{u}}\|_{\mathbf{H}^{-1}} \\ & \le \norm{
       \bm{\mathcal{P}} (\mathbf{f} - \mathbf{K} \widetilde{\mathbf{u}})
        }_{\mathbf{H}^{-1}} + \|\bm{\mathcal{E}}_Z \bm{\mathcal{L}}_k^{-1}
    \mathbf{Z}_k^\top \mathbf{f}\|_{\mathbf{H}^{-1}}+\norm{ \bm{\mathcal{E}}_\mathcal{P}
    \widetilde{\mathbf{u}}}_{\mathbf{H}^{-1}} \\ &\le \norm{
        \bm{\mathcal{P}} (\mathbf{f} - \mathbf{K} \widetilde{\mathbf{u}})
        }_{\mathbf{H}^{-1}} + \varepsilon_\mathrm{svd} \sqrt{k} \paren[\big]{
        (1+\widetilde{\sigma}_k^2)^{-\frac{1}{2}}
        \norm{\mathbf{f}}_{\mathbf{H}^{-1}} + \sqrt{2} \norm{\mathbf{H}
        \widetilde{\mathbf{u}}}_{\mathbf{H}^{-1}} }.\qedhere
    \end{align*}   
\end{proof}

\section{TriCG with deflated restarting} 
\label{sec:tricgdr}

Based on the gSSY process with deflated restarting, we propose a new method
called TriCG with deflated restarting (TriCG-DR) in this section. Let $p$ and
$k$ denote the maximum size of the subspace dimension and the number of desired
approximate elliptic singular vectors, respectively. TriCG-DR($p,k$) incorporates a recycling
mechanism. For the first cycle, the recurrences for the iterates are the same as
that of TriCG. From the second and latter cycles, the recurrences for the
iterates of TriCG-DR($p,k$) are different from that of TriCG. We present
the recurrences for the second cycle, and the same recurrences holds for the
latter cycles.

At the end of the first cycle, we have the relation \eqref{SSY_relations2},
and it is used for the next cycle. We construct the new basis vector matrices
$\widetilde{\mathbf{U}}_{k+1}$ and $\widetilde{\mathbf{V}}_{k+1}$ via
\eqref{eq:new_basis_gssy-dr}. Continuing to generate the basis vectors by
gSSY-DR($p,k$) yields a new relation \eqref{new_deflated_decomposition}.
Let $\widetilde{\mathbf{x}}_k = \mathbf{x}_p$ and $\widetilde{\mathbf{y}}_k =
\mathbf{y}_p$ be the initial iterates of the second cycle, where $\mathbf{x}_p$
and $\mathbf{y}_p$ be the $p$th iterates obtained at the end of the first cycle.
The other $p-k$ iterates of the second cycle are
\[
    \begin{bmatrix}
        \widetilde{\mathbf{x}}_j \\ \widetilde{\mathbf{y}}_j
    \end{bmatrix} = 
    \begin{bmatrix}
        \widetilde{\mathbf{x}}_k \\ \widetilde{\mathbf{y}}_k
    \end{bmatrix} + 
    \widetilde{\mathbf{W}}_j \widetilde{\mathbf{z}}_j,\quad k+1 \le j \le p,
\]
where
\[
    \widetilde{\mathbf{W}}_{j} = 
    \begin{bmatrix}
        \widetilde{\mathbf{U}}_{j} \\ & \widetilde{\mathbf{V}}_{j}
    \end{bmatrix} \mathbf{P}_{j},\quad
    \mathbf{P}_j =
    \begin{bmatrix}
        \mathbf{e}_1 & \mathbf{e}_{j+1} & \cdots & \mathbf{e}_i &
        \mathbf{e}_{j+i} & \cdots & \mathbf{e}_j & \mathbf{e}_{2j}
    \end{bmatrix} \in \mathbb{R}^{2j \times 2j},
\] and $\widetilde{\mathbf{z}}_j$ satisfies the Galerkin condition $\widetilde{\mathbf{r}}_j \perp
\range(\widetilde{\mathbf{W}}_j)$.
Recall from \eqref{eq:rj_tricg} that 
\[
    \widetilde{\mathbf{r}}_k = \mathbf{r}_p = \mathbf{H} \mathbf{W}_{p+1}
    \bigl(\widetilde{\beta}_1 \mathbf{e}_{2p+1} + \widetilde{\gamma}_1
    \mathbf{e}_{2p+2} \bigr) = 
    \mathbf{H}
    \begin{bmatrix}
        \mathbf{u}_{p+1} & \mathbf{0} \\ \mathbf{0} & \mathbf{v}_{p+1}
    \end{bmatrix}
    \begin{bmatrix}
        \widetilde{\beta}_1 \\ \widetilde{\gamma}_1
    \end{bmatrix}
\]
where $\widetilde{\beta}_1 = -\beta_{p+1} \xi_{2p}$ and $\widetilde{\gamma}_1 =
-\gamma_{p+1} \xi_{2p-1}$. From
\eqref{eq:new_basis_gssy-dr}, we observe that 
\[
    \widetilde{\mathbf{U}}_j \mathbf{e}_{k+1} = \mathbf{u}_{p+1}, \quad
    \widetilde{\mathbf{V}}_j \mathbf{e}_{k+1} = \mathbf{v}_{p+1}, \quad
    j \ge k+1.
\]
Thus,
\[
    \widetilde{\mathbf{r}}_k = \mathbf{H} \widetilde{\mathbf{W}}_j
    (\widetilde{\beta}_1 \mathbf{e}_{2k+1} + \widetilde{\gamma}_1
    \mathbf{e}_{2k+2}),\quad j \ge k+1.
\] From \eqref{new_deflated_decomposition}, the corresponding
residual
\begin{align*}
    \widetilde{\mathbf{r}}_j 
    &= 
    \begin{bmatrix}
        \mathbf{b} \\ \mathbf{c}
    \end{bmatrix} -
    \begin{bmatrix}
        \mathbf{M} & \mathbf{A} \\ \mathbf{A}^\top & -\mathbf{N}
    \end{bmatrix}
    \biggl(
        \begin{bmatrix}
            \widetilde{\mathbf{x}}_k \\ \widetilde{\mathbf{y}}_k
        \end{bmatrix} + 
        \widetilde{\mathbf{W}}_j \widetilde{\mathbf{z}}_j
    \biggr) \\
    &= \widetilde{\mathbf{r}}_k - \mathbf{H} \widetilde{\mathbf{W}}_{j+1}
    \widetilde{\mathbf{S}}_{j+1, j} \widetilde{\mathbf{z}}_j \\
    &= \mathbf{H} \widetilde{\mathbf{W}}_{j+1} 
    \bigl(
        \widetilde{\beta}_1 \mathbf{e}_{2k+1} + \widetilde{\gamma}_1
        \mathbf{e}_{2k+2} - \widetilde{\mathbf{S}}_{j+1, j}
        \widetilde{\mathbf{z}}_j
    \bigr),
\end{align*} 
where 
\[
    \widetilde{\mathbf{S}}_{j+1, j} = 
    \begin{bmatrix}
        \widetilde{\bm{\Omega}}_1 	&	 		& 	&	\widetilde{\bm{\Psi}}_2 		\\
                                    & 	\ddots	&	&	\vdots					\\
                                    &			&	\ddots	& \widetilde{\bm{\Psi}}_{k+1} \\
        \widetilde{\bm{\Psi}}_2^\top	&	\dots	&	\widetilde{\bm{\Psi}}_{k+1}^\top	&	\widetilde{\bm{\Omega}}_{k+1} & \widetilde{\bm{\Psi}}_{k+2} \\
                                        & & & \widetilde{\bm{\Psi}}_{k+2}^\top & \widetilde{\bm{\Omega}}_{k+2} & \ddots \\
                                        & & & & \ddots & \ddots & \widetilde{\bm{\Psi}}_j \\
                                        & & & &            & \ddots     & \widetilde{\bm{\Omega}}_j \\
                                        & & & &            &			& \widetilde{\bm{\Psi}}_{j+1}^\top \\
    \end{bmatrix}
\]
with 
\[
    \widetilde{\bm{\Omega}}_j = 
    \begin{bmatrix}
        1 & \widetilde{\alpha}_j \\  \widetilde{\alpha}_j & -1
    \end{bmatrix},\quad
    \widetilde{\bm{\Psi}}_j = 
    \begin{bmatrix}
        0 & \widetilde{\gamma}_j \\ \widetilde{\beta}_j & 0
    \end{bmatrix}.
\]
The Galerkin condition $\widetilde{\mathbf{r}}_j \perp
\range(\widetilde{\mathbf{W}}_j)$ yields the subproblem
\[
    \widetilde{\mathbf{S}}_j \widetilde{\mathbf{z}}_j = \widetilde{\beta}_1
    \mathbf{e}_{2k+1} + \widetilde{\gamma}_1 \mathbf{e}_{2k+2},\quad 
    \widetilde{\mathbf{z}}_j =
    \begin{bmatrix}
        \xi_1 & \xi_2 & \cdots & \xi_{2j}
    \end{bmatrix}^\top,\quad j \geq k + 1,
\]
where $\widetilde{\mathbf{S}}_j$ is the leading $2j \times 2j$ submatrix of
$\widetilde{\mathbf{S}}_{j+1, j}$. Consider the $\mathrm{LDL}^\top$
factorization of $\widetilde{\mathbf{S}}_j = \mathbf{L}_j \mathbf{D}_j
\mathbf{L}_j^\top$, where
\[
    \mathbf{L}_j = 
    \begin{bmatrix}
        \bm{\Delta}_1 \\
        & \ddots \\
        &			& \ddots \\
        \bm{\Gamma}_2 	& \dots 	& \bm{\Gamma}_{k+1} 	& \bm{\Delta}_{k+1} \\
                        &  			&  				& \bm{\Gamma}_{k+2} 	& \bm{\Delta}_{k+2} \\
                        &			&				&				& \ddots 	& \ddots \\
                        &			&				&				&			& \bm{\Gamma}_j	& \bm{\Delta}_j \\
    \end{bmatrix}, \quad
    \bm{\Delta}_j = 
    \begin{bmatrix}
        1 \\ \delta_j & 1
    \end{bmatrix},\quad
    \bm{\Gamma}_j = 
    \begin{bmatrix}
        & \sigma_j \\ \eta_j & \lambda_j
    \end{bmatrix},
\]
and $\mathbf{D}_j = \mathrm{diag} \{d_1, d_2, \dots, d_{2j}\}$. By comparing
both sides of $\widetilde{\mathbf{S}}_j = \mathbf{L}_j \mathbf{D}_j
\mathbf{L}_j^\top$, we deduce that, for $\ell = 1, 2, \dots, k$,
\begin{subequations} \label{eq:update_Lk}
    \begin{align}
        d_{2\ell-1} &= 1, \\
        \delta_\ell &= \widetilde{\alpha}_\ell / d_{2\ell-1}, \\
        d_{2\ell} &= -1 - d_{2\ell-1} \delta_\ell^2, \\
        \eta_{\ell+1} &= \widetilde{\gamma}_{\ell+1} / d_{2\ell-1}, \\
        \sigma_{\ell+1} &= \widetilde{\beta}_{\ell+1} / d_{2\ell}, \\
        \lambda_{\ell+1} &= -d_{2\ell-1} \delta_\ell \eta_{\ell+1} / d_{2\ell},
    \end{align}
\end{subequations}
and 
\begin{subequations} 
    \begin{align}
        d_{2k+1} &= 1 - \sum_{\ell=1}^{k} d_{2\ell} \sigma_{\ell+1}^2, \\
        \delta_{k+1} &= (\widetilde{\alpha}_{k+1} - \sum_{\ell=1}^{k} d_{2\ell}
        \lambda_{\ell+1} \sigma_{\ell+1}) / d_{2k+1}, \\
        d_{2k+2} &= -1 - \sum_{\ell=1}^{k} (d_{2\ell-1} \eta_{\ell+1}^2 + d_{2\ell}
        \lambda_{\ell+1}^2) - d_{2k+1} \delta_{k+1}^2.
    \end{align}
\end{subequations}
For $j \ge k+2$, the recurrences for $\mathbf{L}_j$ and $\mathbf{D}_j$ are the same
as that of TriCG, i.e., \eqref{eq:ldl}.
We update the solution $\mathbf{p}_j = \begin{bmatrix} \pi_1 & \pi_2 & \cdots &
\pi_{2j} \end{bmatrix}^\top$ of $\mathbf{L}_j \mathbf{D}_j \mathbf{p}_j =
\widetilde{\beta}_1 \mathbf{e}_{2k+1} + \widetilde{\gamma}_1 \mathbf{e}_{2k+2}$
rather than computing $\widetilde{\mathbf{z}}_j$. The components of
$\mathbf{p}_j$ are updated by
\begin{equation} 
    \pi_1 = \dots = \pi_{2k} = 0, \quad
    \pi_{2k+1} = \widetilde{\beta}_1 / d_{2k+1}, \quad
    \pi_{2k+2} = (\widetilde{\gamma}_1 - \widetilde{\beta}_1 \delta_{k+1}) / d_{2k+2}.
\end{equation}
For $j \ge k+2$, the recurrences of $\pi_{2j-1}$ and $\pi_{2j}$ are the same as that
of TriCG, i.e., \eqref{eq:pk}. From $\mathbf{L}_j^\top \widetilde{\mathbf{z}}_j =
\mathbf{p}_j$, we obtain
\begin{equation} \notag
    \xi_{2j-1} = \pi_{2j-1} - \delta_{j} \pi_{2j}, \quad
    \xi_{2j} = \pi_{2j}, \quad j \ge k+1.
\end{equation}
Let
\[
    \mathbf{G}_j = \mathbf{W}_j \mathbf{L}_j^{-\top},\quad \mathbf{G}_j = 
    \begin{bmatrix}
        \mathbf{G}_j^x \\ \mathbf{G}_j^y
    \end{bmatrix} = 
    \begin{bmatrix}
        \mathbf{g}_1^x & \dots & \mathbf{g}_{2j}^x \\
        \mathbf{g}_1^y & \dots & \mathbf{g}_{2j}^y
    \end{bmatrix}.
\]
Then, we have the recurrences, for $\ell = 1,\dots, k$, 
\begin{subequations}
    \begin{alignat}{2}
        \mathbf{g}_{2\ell-1}^x &= \widetilde{\mathbf{u}}_\ell,\quad &
        \mathbf{g}_{2\ell-1}^y &= \mathbf{0}, \\
        \mathbf{g}_{2\ell}^x &= -\delta_\ell \widetilde{\mathbf{u}}_\ell, \quad &
        \mathbf{g}_{2\ell}^y &= \widetilde{\mathbf{v}}_\ell,
    \end{alignat}
\end{subequations}
and 
\begin{subequations} \label{eq:update_gk}
    \begin{align}
        \mathbf{g}_{2k+1}^x &= \widetilde{\mathbf{u}}_{k+1} + \sum_{\ell=1}^{k}
        \sigma_{\ell+1} \delta_{\ell} \widetilde{\mathbf{u}}_\ell, \\
        \mathbf{g}_{2k+1}^y &= -\sum_{\ell=1}^{k} \sigma_{\ell+1}
        \widetilde{\mathbf{v}}_\ell, \\
        \mathbf{g}_{2k+2}^x &= -\delta_{k+1} \mathbf{g}_{2k+1}^x -
        \sum_{\ell=1}^{k} (\eta_{\ell+1} - \lambda_{\ell+1} \delta_\ell)
        \widetilde{\mathbf{u}}_\ell, \\
        \mathbf{g}_{2k+2}^y &= \widetilde{\mathbf{v}}_{k+1} - \delta_{k+1}
        \mathbf{g}_{2k+1}^y - \sum_{\ell=1}^{k} \lambda_{\ell+1}
        \widetilde{\mathbf{v}}_\ell.
    \end{align}
\end{subequations}
For $j \ge k+2$, the recurrences of $\mathbf{G}_j$ are the same as that of
TriCG, i.e., \eqref{eq:g}. Then, for $j \ge k+1$, the
approximate solution is updated by
\begin{subequations}\label{eq:update_2xy}\begin{align}
    \widetilde{\mathbf{x}}_{j} &= \widetilde{\mathbf{x}}_k +
    \mathbf{G}_{j}^x \mathbf{p}_{j} = \widetilde{\mathbf{x}}_{j-1} + \pi_{2j-1}
    \mathbf{g}_{2j-1}^x + \pi_{2j} \mathbf{g}_{2j}^x, \\
    \widetilde{\mathbf{y}}_{j} &= \widetilde{\mathbf{y}}_k +
    \mathbf{G}_{j}^y \mathbf{p}_{j} = \widetilde{\mathbf{y}}_{j-1} + \pi_{2j-1}
    \mathbf{g}_{2j-1}^y + \pi_{2j} \mathbf{g}_{2j}^y.
\end{align}
\end{subequations}
Obviously, the recurrences \eqref{eq:update_Lk}--\eqref{eq:update_gk}
coincide with that of TriCG if $k = 0$.

When the desired approximate elliptic singular vectors have reached sufficient accuracy, but the approximate solution to the SQD linear system has not yet attained the specified precision, in order to save computational cost, we stop updating the approximate elliptic singular triplets and use the recurrence \eqref{eq:update_2xy} to compute the approximate solution until the desired precision is achieved. Specifically, TriCG-DR($p,k$) has two distinct stages: the restarting stage and the non-restarting stage.
\begin{enumerate}
   \item \emph{The restarting stage:} When the desired $k$ approximate elliptic singular
       triplets do not satisfy the criteria \eqref{eq:svd_stop_criteria} and
       the residual norm does not reduce to the given tolerance, we employ gSSY-DR($p,k$) to update the approximate elliptic singular
       triplets and use \eqref{eq:update_2xy} with $k+1\leq j\leq p$ to compute the approximate solution.
   \item \emph{The non-restarting stage:} If the desired $k$ approximate elliptic
       singular triplets satisfy the criteria \eqref{eq:svd_stop_criteria} but
       the residual norm has not yet reduced to the given tolerance, we use \eqref{eq:update_2xy} with $j\geq k+1$ to compute the approximate solution until either the user-defined maximum number of iteration is exceeded or an sufficient accurate approximate solution is obtained. 
\end{enumerate}
We summarize the implementations of TriCG-DR($p,k$) in \Cref{alg:tricgdr}. Note that some reorthogonalization steps (see lines 16 and 18) are used to control the rounding errors.

\begin{algorithm}[htbp]
    \caption{TriCG with deflated restarting: TriCG-DR($p,k$)}
    \label{alg:tricgdr}

    \KwIn{$\mathbf{A}$, $\mathbf{M}$, $\mathbf{N}$, $\mathbf{b}$, $\mathbf{c}$.
        $p$--number of maximum subspace dimension; 
        $k$--number of desired elliptic singular triplets;
        \texttt{maxcycle}--maximum number of cycles;
        \texttt{maxit}--maximum number of iterations for the non-restarting
        stage;
        \texttt{tol}--tolerance for approximate solutions; 
        $\varepsilon_\mathrm{svd}$--tolerance for approximate elliptic singular triplets.
    }

    \KwOut{Approximate solution $\mathbf{x}$ and $\mathbf{y}$}

    $\beta_1 \mathbf{M} \mathbf{u}_1 = \mathbf{b}$, $\gamma_1 \mathbf{N}
    \mathbf{v}_1 = \mathbf{c}$ \;

    $\mathbf{U}_1 = [\mathbf{u}_1]$, $\mathbf{V}_1 = [\mathbf{v}_1]$\;

    $k_{\rm aug} = k$, $k = 0$ 
    \tcc*[r]{Set the number of deflation vectors to zero at the first cycle}

    \texttt{conv\_sv = false}, 
    $\mathtt{inner} = p$, $\mathbf{x}_0 =
    \mathbf{0}$, $\mathbf{y}_0 = \mathbf{0}$ 
    \tcc*[r]{\texttt{conv\_sv} checks the convergence of elliptic singular
    values}

    \For(\tcc*[f]{Outer cycle}){$\mathtt{outerit} = 1, 2, \dots, \mathtt{maxcycle}$}
    {
        Compute $\alpha_{k+1}$, $\beta_{k+2}$, $\gamma_{k+2}$,
        $\mathbf{u}_{k+2}$, $\mathbf{v}_{k+2}$ via lines
        \ref{alg_line:gssy_outer_q}--\ref{alg_line:gssy_outer_v} of
        \Cref{alg:approx_svd}\;

        Compute $d_{2k+1}$, $d_{2k+2}$, $\delta_{k+1}$, $\pi_{2k+1}$,
        $\pi_{2k+2}$, $\mathbf{g}_{2k+1}^x$, $\mathbf{g}_{2k+1}^y$,
        $\mathbf{g}_{2k+2}^x$, $\mathbf{g}_{2k+2}^y$ via
        \eqref{eq:update_Lk}--\eqref{eq:update_gk}\;

        $\mathbf{x}_{k+1} = \mathbf{x}_k + \pi_{2k+1} \mathbf{g}_{2k+1}^x +
        \pi_{2k+2} \mathbf{g}_{2k+2}^x$\;

        $\mathbf{y}_{k+1} = \mathbf{y}_k + \pi_{2k+1} \mathbf{g}_{2k+1}^y +
        \pi_{2k+2} \mathbf{g}_{2k+2}^y$\;

        $\xi_{2k+1} = \pi_{2k+1} - \delta_{k+1} \pi_{2k+2},\ \xi_{2k+2} =
        \pi_{2k+2}$\;

        $\|\mathbf{r}_{k+1}\|_{\mathbf{H}^{-1}} = (\gamma_{k+2}^2 \xi_{2k+1}^2 +
        \beta_{k+2}^2 \xi_{2k+2}^2)^{1/2}$\;

        \lIf(
            \tcc*[f]{Update $\mathbf{T}$ only when the elliptic singular values
            do not converge}
        ){$\mathtt{!conv\_sv}$}{$\mathbf{T}_{k+1,k+1} = \alpha_{k+1}$}

        \For(
            \tcc*[f]{Inner iteration}
        ){$j = k+2, k+3, \dots, \mathtt{inner}$}
        {
            Compute $\alpha_j$, $\mathbf{Mu}$, $\mathbf{Nv}$ via lines
            \ref{alg_line:gssy_inner_q}--\ref{alg_line:gssy_inner_v} of
            \Cref{alg:approx_svd} \;

            \eIf{$\mathtt{!conv\_sv}$}
            {
                Update $\mathbf{T}$, and reorthogonalize $\mathbf{Mu}$ and
                $\mathbf{Nv}$ via lines
                \ref{alg_line:gssy_update_T}--\ref{alg_line:gssy_reorth_v} of
                \Cref{alg:approx_svd}
            }{
                Only reorthogonalize $\mathbf{Mu}$ and $\mathbf{Nv}$ with respect to the
                converged elliptic singular \hspace*{-1ex}vectors: $\mathbf{Mu} = (\mathbf{I}
                - \mathbf{MU}_k \mathbf{U}_k^\top) \mathbf{Mu}$, $\mathbf{Nv} =
                (\mathbf{I} - \mathbf{NV}_k \mathbf{V}_k^\top) \mathbf{Nv}$ 
            }

            $\beta_{j+1} = \|\mathbf{u}\|_{\mathbf{M}}$, $\gamma_{j+1} =
            \|\mathbf{v}\|_{\mathbf{N}}$\;

            $\mathbf{u}_{j+1} = \mathbf{u} / \beta_{j+1}$, $\mathbf{v}_{j+1} =
            \mathbf{v} / \gamma_{j+1}$\;
            
            Compute $\eta_j$, $\sigma_j$, $\lambda_j$, $d_{2j-1}$, $\delta_j$,
            $d_{2j}$, $\pi_{2j-1}$, $\pi_{2j}$, $\mathbf{g}_{2j-1}^x$,
            $\mathbf{g}_{2j-1}^y$, $\mathbf{g}_{2j}^x$, $\mathbf{g}_{2j}^y$ via
            \eqref{eq:ldl}--\eqref{eq:g}\;

            $\mathbf{x}_{j} = \mathbf{x}_{j-1} + \pi_{2j-1} \mathbf{g}_{2j-1}^x
            + \pi_{2j} \mathbf{g}_{2j}^x$\;

            $\mathbf{y}_{j} = \mathbf{y}_{j-1} + \pi_{2j-1} \mathbf{g}_{2j-1}^y
            + \pi_{2j} \mathbf{g}_{2j}^y$\;

            $\xi_{2j-1} = \pi_{2j-1} - \delta_j \pi_{2j},\ \xi_{2j} = \pi_{2j}$\;

            $\|\mathbf{r}_j\|_{\mathbf{H}^{-1}} = (\gamma_{j+1}^2 \xi_{2j-1}^2 +
            \beta_{j+1}^2 \xi_{2j}^2)^{1/2}$\;

            \lIf{$\|\mathbf{r}_j\|_{\mathbf{H}^{-1}} \le \mathtt{tol}$}
            {\textbf{stop}}
        }
        \lIf{$\mathtt{conv\_sv}$}{\textbf{stop}}\tcc*[r]{The maximum number of iterations is exceeded, but the
         residual norm fails to reduce to the given tolerance during the non-restarting
         stage}
         
        $k = k_{\rm aug}$ \tcc*[r]{Recover dimension of augmentation to $k$}
        $\beta_1 = -\beta_{p+1} \xi_{2p}$, $\gamma_1 = -\gamma_{p+1} \xi_{2p-1}$
        \;

        Compute the SVD of $\mathbf{T}$, and store the $k$ desired singular
        triplets in $\widehat{\mathbf{U}}_k,\ \bm{\Sigma}_k$ and
        $\widehat{\mathbf{V}}_k$\;

        Let $\mathbf{U}_k = \mathbf{U}_p \widehat{\mathbf{U}}_k$ and
        $\mathbf{V}_k = \mathbf{V}_p \widehat{\mathbf{V}}_k$\;

        Check the number of converged elliptic singular triplets
        \texttt{num\_conv\_sv} via lines
        \ref{line:gssy_check_converged_num1}--\ref{line:gssy_check_converged_num2}
        of \hspace*{-1ex}\Cref{alg:approx_svd}\;
        \If(\tcc*[f]{Don't restart if the $k$ approximate elliptic singular values
            converge}){$\mathtt{num\_conv\_sv} = k$}
        {
            \texttt{conv\_sv = true}, \texttt{inner = maxit}
        }

        $\mathbf{U}_{k+1} = \begin{bmatrix} \mathbf{U}_k & \mathbf{u}_{p+1}
            \end{bmatrix}$, $\mathbf{V}_{k+1} = \begin{bmatrix} \mathbf{V}_k &
        \mathbf{v}_{p+1} \end{bmatrix}$\; \smallskip

        $\mathbf{T}_{1:k,1:k} = \bm{\Sigma}_k$, $\mathbf{T}_{1:k,k+1}
        = \gamma_{p+1} \widehat{\mathbf{U}}_k^\top \mathbf{e}_p$,
        $\mathbf{T}_{k+1,1:k} = \beta_{p+1} \mathbf{e}_p^\top
        \widehat{\mathbf{V}}_k$\;
    }
\end{algorithm}

\section{Multiple right-hand sides}\label{sec:mrhs}
We now consider SQD linear systems with multiple right-hand sides. We use the approximate elliptic singular vector matrices $\mbf U_k$ and $\mbf V_k$ generated in TriCG-DR($p,k$) for the solution of the linear system with the first right-hand side to deflate elliptic singular values from the solution of the subsequent right-hand sides. 

Let $\bem\mbf b_i^\top & \mbf c_i^\top\eem^\top$ denote the $i$th  ($i\ge 2$) right-hand side. First, we compute an initial guess  
\begin{align}
    \begin{bmatrix}
        \mathbf{x}_0 \\ \mathbf{y}_0
    \end{bmatrix}
    &= 
    \begin{bmatrix}
        \mathbf{U}_k \\ & \mathbf{V}_k
    \end{bmatrix}
    \biggl(
        \begin{bmatrix}
            \mathbf{U}_k^\top \\ & \mathbf{V}_k^\top
        \end{bmatrix}   
        \begin{bmatrix}
            \mathbf{M} & \mathbf{A} \\ \mathbf{A}^\top & -\mathbf{N}
        \end{bmatrix}
        \begin{bmatrix}
            \mathbf{U}_k \\ & \mathbf{V}_k
        \end{bmatrix}
    \biggr)^{-1}
    \begin{bmatrix}
        \mathbf{U}_k^\top \\ & \mathbf{V}_k^\top
    \end{bmatrix}   
    \begin{bmatrix}
        \mathbf{b}_i \\ & \mathbf{c}_i
    \end{bmatrix} \notag \\ 
    &=
    \begin{bmatrix}
        \mathbf{U}_k \\ & \mathbf{V}_k
    \end{bmatrix}
    \begin{bmatrix}
        \mathbf{I}_k & \mathbf{T}_k \\ \mathbf{T}_k^\top & -\mathbf{I}_k
    \end{bmatrix}^{-1}
    \begin{bmatrix}
        \mathbf{U}_k^\top \\ & \mathbf{V}_k^\top
    \end{bmatrix}   
    \begin{bmatrix}
        \mathbf{b}_i \\ \mathbf{c}_i
    \end{bmatrix}. \label{eq:x0}
\end{align} 
    Define 
    \[
        \begin{bmatrix}
            \mathbf{d}_x \\ \mathbf{d}_y
        \end{bmatrix}
        = 
        \begin{bmatrix}
            \mathbf{I}_k & \mathbf{T}_k \\ \mathbf{T}_k^\top & -\mathbf{I}_k
        \end{bmatrix}^{-1}
        \begin{bmatrix}
            \mathbf{U}_k^\top \\ & \mathbf{V}_k^\top
        \end{bmatrix}   
        \begin{bmatrix}
            \mathbf{b}_i \\ \mathbf{c}_i
        \end{bmatrix}.
    \]
   Then the TriCG method is used for 
\begin{equation} \label{eq:dtricg_sys}
    \begin{bmatrix}
        \mathbf{M} & \mathbf{A} \\ \mathbf{A}^\top & -\mathbf{N}
    \end{bmatrix}
    \begin{bmatrix}
        \mathbf{x} - \mathbf{x}_0 \\ \mathbf{y} - \mathbf{y}_0
    \end{bmatrix} = 
    \begin{bmatrix}
        \mathbf{r}_0^x \\ \mathbf{r}_0^y
    \end{bmatrix},
\end{equation} with $$ 
    \begin{bmatrix}
        \mathbf{r}_0^x \\ \mathbf{r}_0^y
    \end{bmatrix}=\begin{bmatrix}
        \mathbf{b}_i \\ \mathbf{c}_i
    \end{bmatrix} - 
    \begin{bmatrix}
        \mathbf{MU}_{k+1} \\ & \mathbf{NV}_{k+1}
    \end{bmatrix}
    \begin{bmatrix}
        \mathbf{I}_{k+1,k} & \mathbf{T}_{k+1,k} \\
        \mathbf{T}_{k,k+1}^\top & -\mathbf{I}_{k+1,k}
    \end{bmatrix}
    \begin{bmatrix}
        \mathbf{d}_x \\ \mathbf{d}_y
    \end{bmatrix}.$$
We call the resulting method deflated TriCG (D-TriCG) since it deflates out
partial spectral information before applying TriCG. It is closely related to the
deflated CG method in \cite{abdelrehim2010deflated}. 

We provide a concise summary of the TriCG-DR+D-TriCG framework for solving SQD linear systems with multiple right-hand sides, as outlined below:
    \begin{enumerate}
        \item For the first right-hand side, TriCG-DR is used to compute the solution and generate the desired $k$ approximate elliptic singular vectors. 
        \item For the subsequent right-hand sides, define the initial guess by
            \eqref{eq:x0}, and then solve \eqref{eq:dtricg_sys} by TriCG.     \end{enumerate}
We would like to point out that reorthogonalizing the computed basis vectors in TriCG against to the $k$ approximate elliptic singular vectors generated during the solution process for the first right-hand side is useful for controlling the rounding errors. 

\section{Numerical experiments} \label{sec:experiment}

In this section, we compare the performance of TriCG-DR and TriCG. Both
algorithms stop as soon as they either reach the maximum number of iterations or
the residual norm $\|\mathbf{r}_k\|_{\mathbf{H}^{-1}}$ falls below the tolerance level $\mathtt{tol}$. All experiments are performed using MATLAB R2025b on a MacBook Air equipped with an Apple M3 chip, 16 GB of memory, and running macOS Tahoe 26.1. The MATLAB scripts to reproduce the results in this section are available at https://github.com/kuidu/tricgdr. For all experiments, the residual norms are computed exactly for a fair comparison.  

We begin with a synthetic example where $\mathbf{M} = \mathbf{I}$,
$\mathbf{N} = \mathbf{I}$, and $\mathbf{A}$ is a diagonal matrix of size $2060 \times 2060$
generated using the following MATLAB script:
\begin{verbatim}
    A = [linspace(0, 800, 2000), linspace(1e3, 1e5, 60)]';
    m = length(A); n = m;
    A = spdiags(A, 0, m, n);
\end{verbatim}
The right-hand vector is generated randomly. It is clear that $\mathbf{A}$ has $60$ large singular values lying in the interval
$[10^3, 10^5]$. In this experiment, we select the $k$ largest singular triplets as the desired ones and investigate the impact of varying $k$ on the convergence behavior of TriCG-DR. The parameters are configured as follows: $k$ is sequentially set to $20, 40$, and $60$, with $p = k + 80$, the convergence tolerance for approximate solutions $\mathtt{tol}$ is set to $10^{-8}$, the convergence tolerance for approximate singular  triplets $\varepsilon_\mathrm{svd}$ is set to $10^{-10}$, the maximum number of cycle $\mathtt{maxcycle}$ is set to $80$, and the maximum number of iterations $\mathtt{maxit}$ is set to $40000$. The convergence histories of TriCG and TriCG-DR($p,k$) are
displayed in \Cref{fig:fig1}. 
 For all tested values of $k$, TriCG-DR consistently demonstrates superior performance compared to TriCG. Notably, as $k$ increases, TriCG-DR exhibits accelerated convergence, highlighting the benefit of incorporating a sufficient large deflation subspace into the algorithm.

\begin{figure}[htbp]
    \centering
    \includegraphics[width=0.48\linewidth]{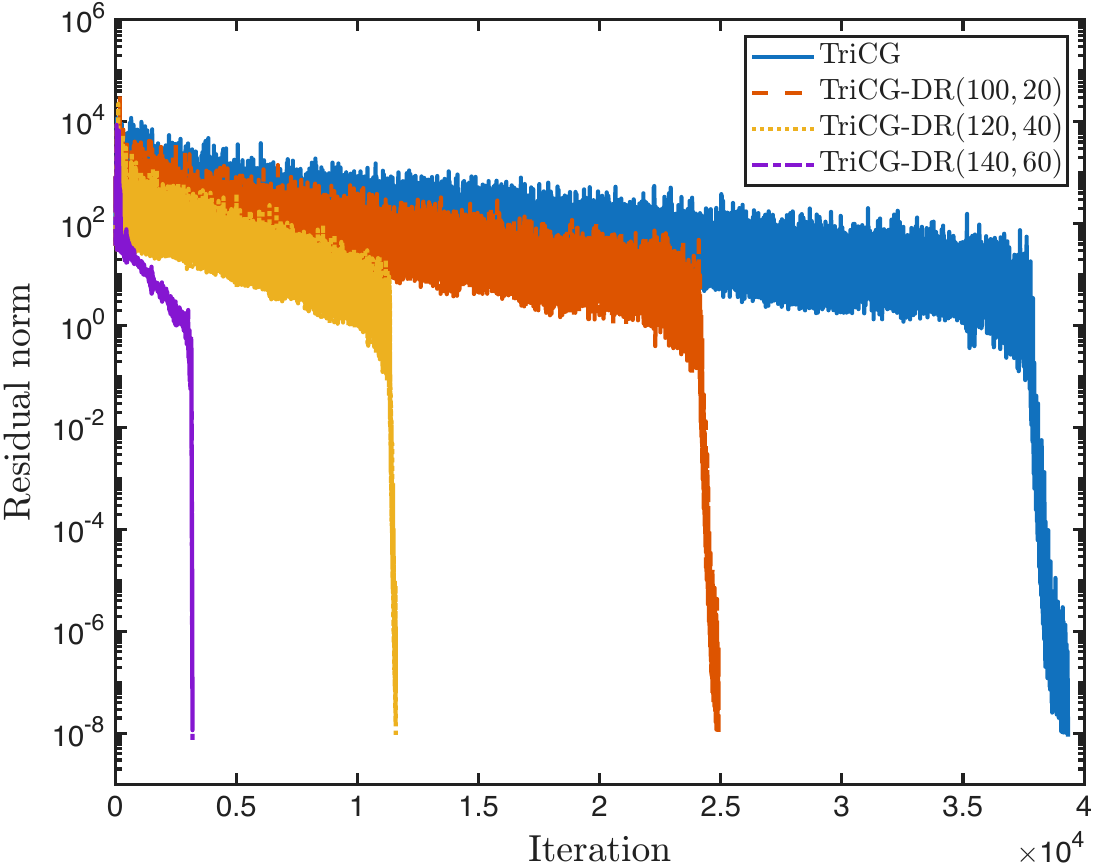}
    \caption{The convergence histories of TriCG and TriCG-DR under varying dimensionality of deflation subspaces.}
\label{fig:fig1}
\end{figure}

In the second experiment, we employ square matrices from the SuiteSparse Matrix Collection
\cite{davis2011university} to serve as the matrix $\mathbf{A}$ in \eqref{eq:problem} and set $\mbf M=\mbf I$ and $\mbf N=\mbf I$. The right-hand vectors $\mathbf{b} =
\mathbf{e} / \sqrt{m}$ and $\mathbf{c} = \mathbf{e} / \sqrt{n}$, where
$\mathbf{e} = \begin{bmatrix} 1 & 1 & \cdots & 1 \end{bmatrix}^\top$. The parameters are configured as follows: $\mathtt{tol} = 10^{-8}$, $\varepsilon_\mathrm{svd} = 10^{-10}$, $\mathtt{maxcycle} = 10$, and $\mathtt{maxit} = 80000$.  We select the $k$ largest singular triplets as the desired ones. The matrix specifications, computational runtimes of TriCG and TriCG-DR, along with the selected values of the TriCG-DR parameters $p$ and $k$, are presented in \Cref{tab:tab1}. The convergence histories of TriCG and TriCG-DR are displayed in \Cref{fig:fig2}. Notably, TriCG-DR demonstrates a significant reduction in iteration counts compared to TriCG, achieving an approximate $1.7\times$ to $3.8\times$ speedup in CPU time.

\begin{table}[htbp]
    \centering
    \caption{The information of square matrices from the SuiteSparse Matrix Collection, runtime of TriCG and TriCG-DR, and parameters $p$ and $k$.}
    \label{tab:tab1}
    \begin{tabular}{|r|r|r|r|rrr|}
        \hline
        \multirow{2}{*}{Matrix} & \multirow{2}{*}{Size} & \multirow{2}{*}{Nnz} &
        TriCG & \multicolumn{3}{r|}{TriCG-DR} \\ \cline{4-7} 
                                &  &  & Time(s) & \multicolumn{1}{r|}{Time(s)} &
        \multicolumn{1}{r|}{$p$} & $k$ \\ \hline
        \texttt{gupta3} & 16783 & 9323427 & 17.55 & \multicolumn{1}{r|}{7.61} &
        \multicolumn{1}{r|}{240} & 120 \\ \hline
        \texttt{g7jac060sc} & 17730 & 183325 & 16.82 &
        \multicolumn{1}{r|}{10.10} & \multicolumn{1}{r|}{60} & 20 \\ \hline
        \texttt{rajat27} & 20640 & 97353 & 24.70 & \multicolumn{1}{r|}{6.42} &
        \multicolumn{1}{r|}{100} & 40 \\ \hline
        \texttt{TSOPF\_RS\_b300\_c2} & 28338 & 2943887 & 30.48 &
        \multicolumn{1}{r|}{17.64} & \multicolumn{1}{r|}{120} & 40 \\ \hline
    \end{tabular}
\end{table}

\begin{figure}[htbp]
    \centering
    \includegraphics[width=0.48\linewidth]{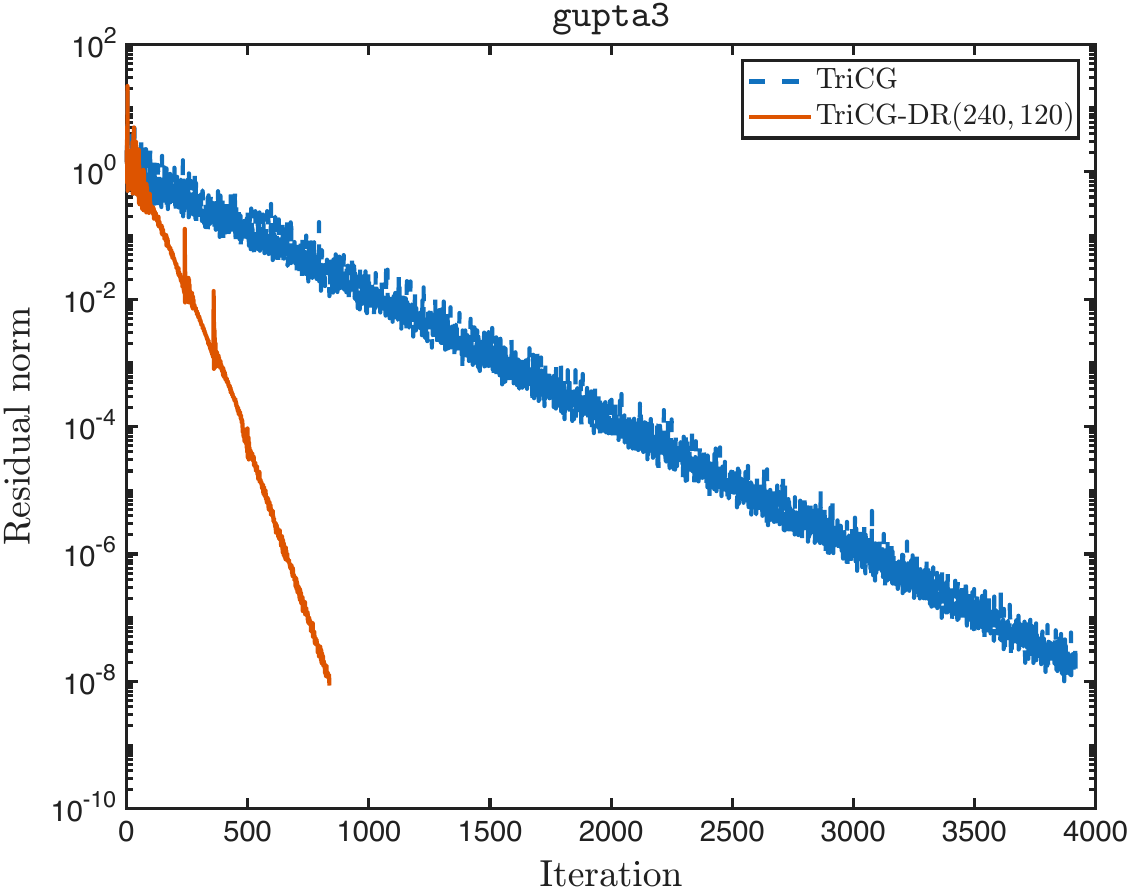}
    \includegraphics[width=0.48\linewidth]{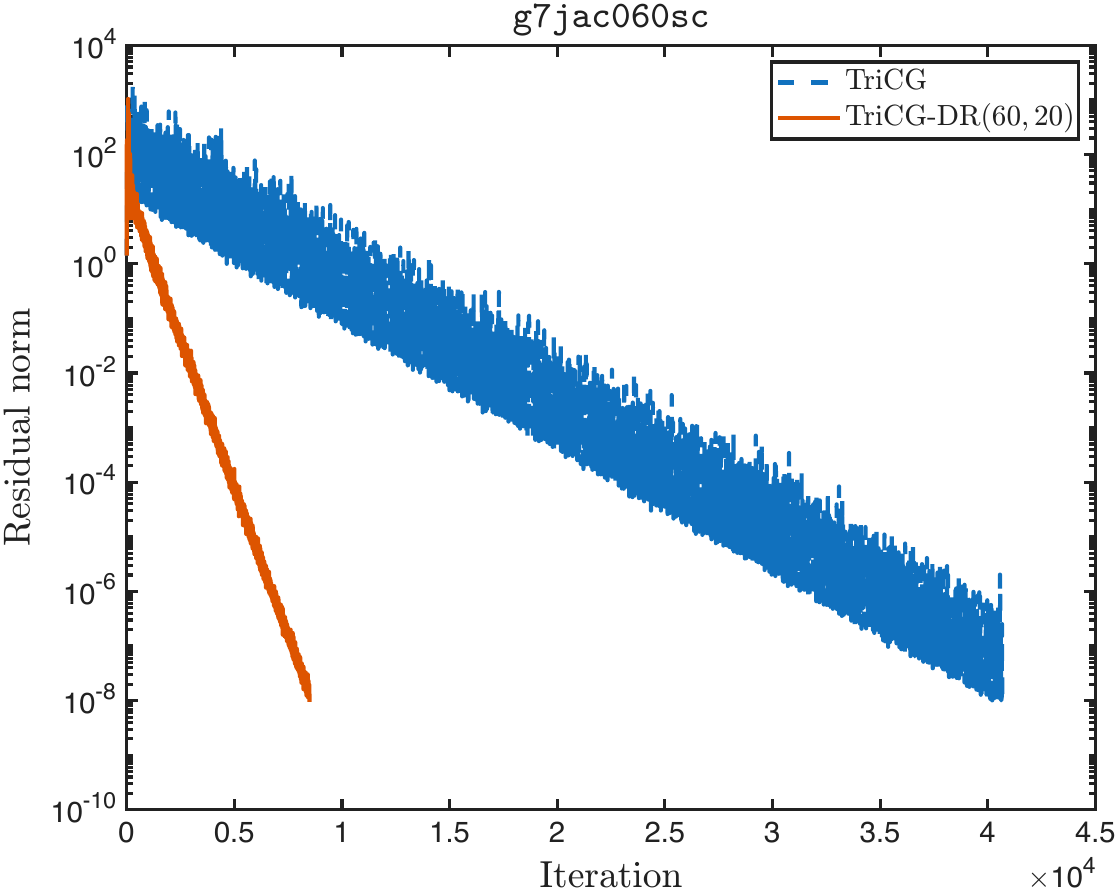}
    \includegraphics[width=0.48\linewidth]{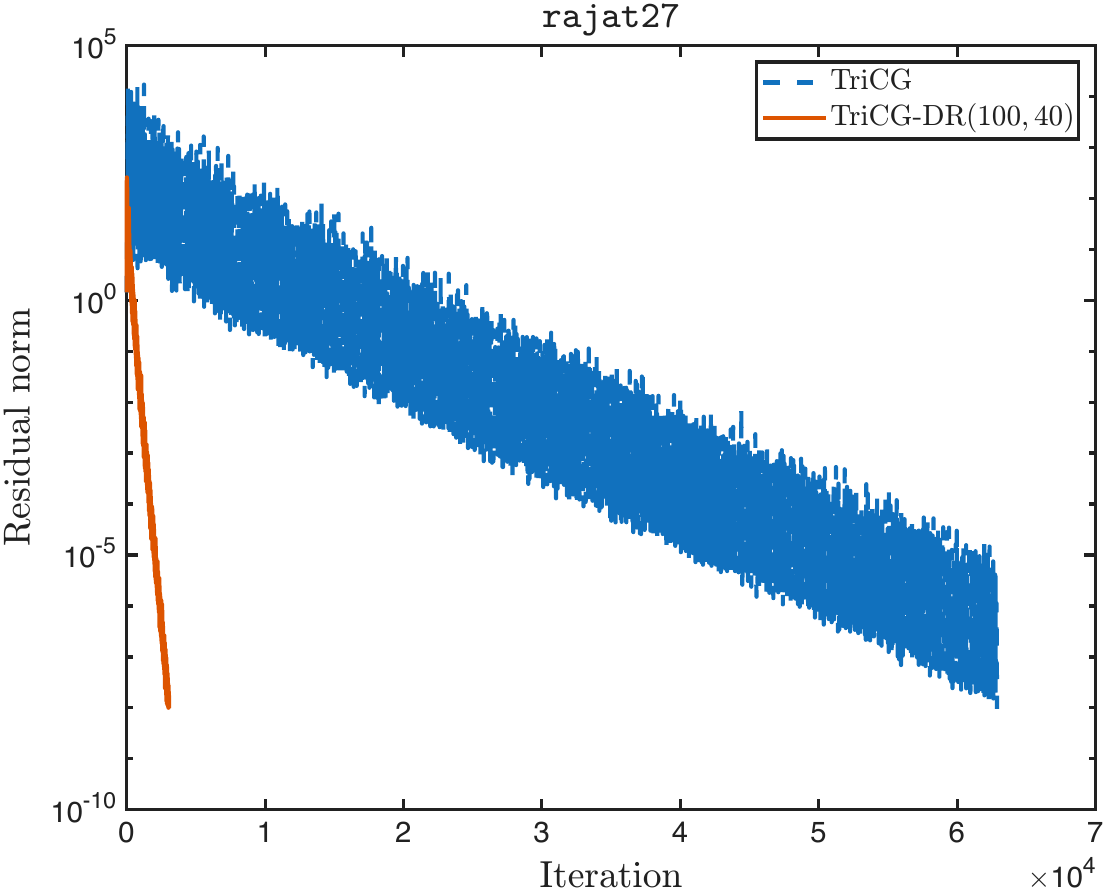}
    \includegraphics[width=0.48\linewidth]{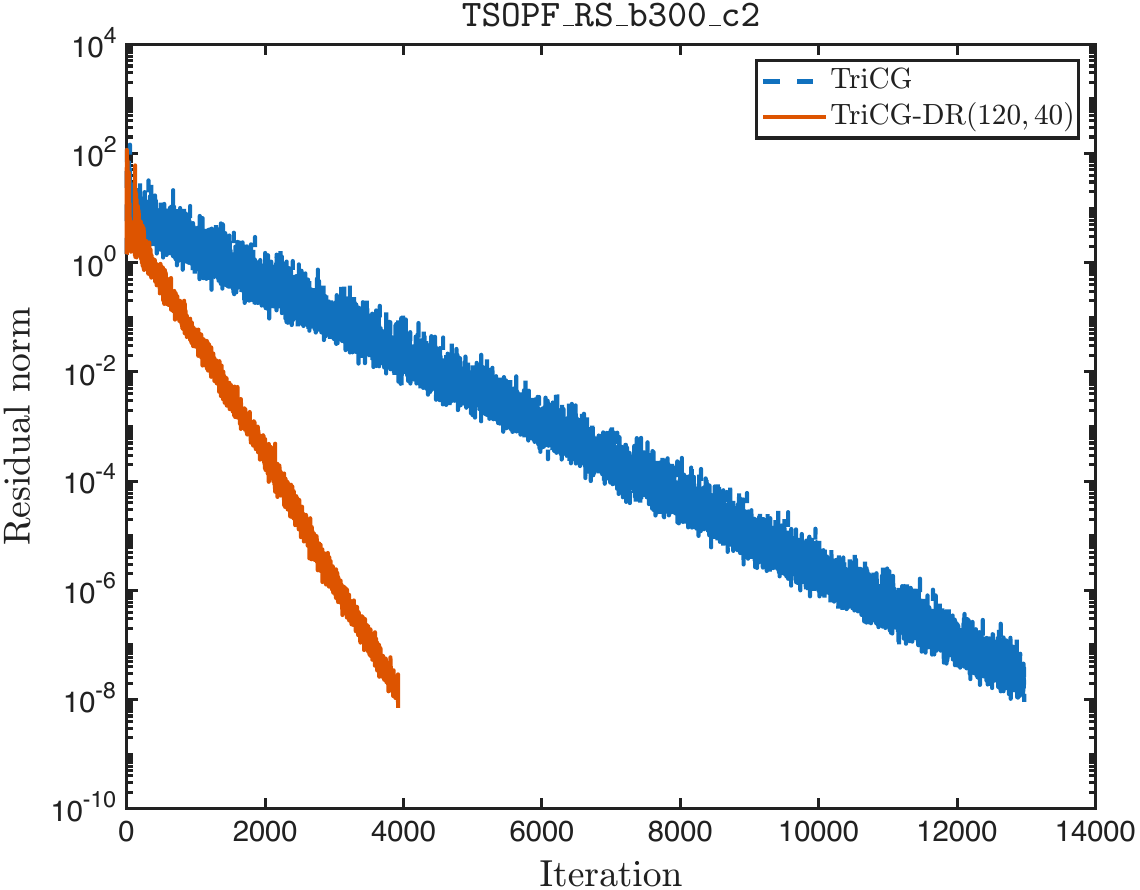}
    \caption{The convergence histories of TriCG and TriCG-DR on the problems
    \texttt{gupta3}, \texttt{g7jac060sc}, \texttt{rajat27}, and
    \texttt{TSOPF\_RS\_b300\_c2}.}
    \label{fig:fig2}
\end{figure}

In the third experiment, we solve SQD linear systems with multiple right-hand
sides using TriCG and TriCG-DR+D-TriCG. 
We set $\mathbf{M} = \mathbf{I}$, $\mathbf{N} = \mathbf{I}$, and use a diagonal matrix $\mathbf{A}$ generated via the following MATLAB script:
\begin{verbatim}
    A = [linspace(0, 100, 1960), linspace(1000, 1020, 40)]';
    m = length(A); n = m;
    A = spdiags(A, 0, m, n);
\end{verbatim} 
It is clear that $\mathbf{A}$ has $40$ large singular values clustered in the interval $[1000, 1020]$. The right-hand sides are randomly generated. 
We investigate the impact of deflation subspace dimensionality on the convergence behavior of D-TriCG. The approximate singular vectors corresponding to the $k$ largest singular values are computed using gSSY-DR($p,k$) with $p=k+40$ and $\varepsilon_\mathrm{svd}=10^{-12}$. When $k = 20$ and $k = 40$, gSSY-DR requires only $3$ and $2$ cycles respectively, achieving singular triplet  errors of $1.76 \times 10^{-13}$ and $1.43 \times 10^{-14}$. For both TriCG and D-TriCG, we set $\mathtt{maxit} = 4000$ and $\mathtt{tol}=10^{-8}$. 
The convergence histories of TriCG and D-TriCG are displayed in
\Cref{fig:fig3_3}. As $k$ increases, D-TriCG demonstrates progressively accelerated convergence. For $k=40$, D-TriCG has a significant performance improvement over the $k=20$ case. This improvement stems from the fact that when $k=20$, the deflation subspace fails to fully eliminate the influence of the cluster of 40 largest singular values. 

\begin{figure}[htbp]
    \centering
    \includegraphics[width=0.48\linewidth]{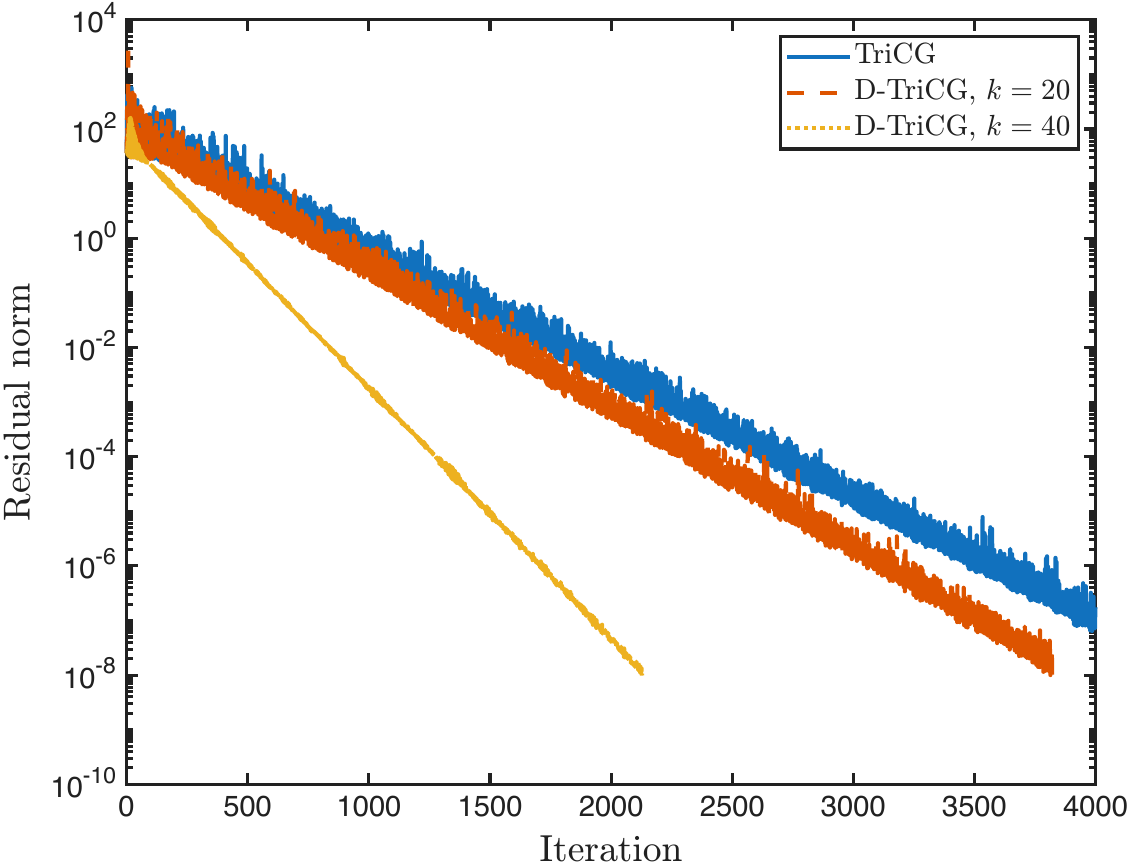}
    \caption{The convergence histories of TriCG and D-TriCG with deflation subspaces of different dimensions.}
    \label{fig:fig3_3}
\end{figure}

At the end of this section, we consider the $Q_1$-$Q_1$ finite element discretization of the unsteady incompressible Stokes equation as in
\cite[Example 3.4]{gueduecue2022non}, which leads to systems of the form
\[
    \begin{bmatrix}
        \bm{\mathcal{M}} & \mathbf{0} \\
        \mathbf{0} & \mathbf{0}
    \end{bmatrix}
    \begin{bmatrix}
        \dot{\mathbf{v}} \\
        \dot{\mathbf{p}}
    \end{bmatrix} = 
    \biggl(
    \begin{bmatrix}
        \mathbf{A}_S & \mathbf{B} \\
        -\mathbf{B}^\top & \mathbf{0}
    \end{bmatrix} - 
    \begin{bmatrix}
        -\mathbf{A}_H & \mathbf{0} \\
        \mathbf{0}^\top & -\mathbf{C}
    \end{bmatrix}
    \biggr)
    \begin{bmatrix}
        \mathbf{v} \\ \mathbf{p}
    \end{bmatrix} + 
    \begin{bmatrix}
        \mathbf{f} \\ \mathbf{g}
    \end{bmatrix}.
\]
The matrices are generated via the IFISS software package \cite{elman2007algorithm} on the benchmark problem
\texttt{channel\_domain}. With the grid parameter set to $8$, we obtain the matrix
$\mathbf{B}$ of size $132098 \times 66049$. For the Stokes problem, $\mathbf{A}_S = \mathbf{0}$, $-\mathbf{A}_H$ is symmetric positive definite, and the stabilization term $-\mathbf{C}$ is symmetric positive semidefinite. Here, we add a small perturbation $10^{-10} \mathbf{I}$ to $-\mathbf{C}$ so that it is positive definite. By left-multiplying $-\mathbf{I}$ with the second block of the linear systems arising from the implicit Euler discretization on a uniform time grid, we obtain a sequence of SQD linear systems of the form: 
\[
    \begin{bmatrix}
        \mathbf{M} & \mathbf{A} \\
        \mathbf{A}^\top & -\mathbf{N}
    \end{bmatrix}
    \begin{bmatrix}
        \mathbf{v}_{i+1} \\ \mathbf{p}_{i+1}
    \end{bmatrix} = 
    \begin{bmatrix}
        \tau \mathbf{f} + \bm{\mathcal{M}} \mathbf{v}_i \\
        -\tau \mathbf{g}
    \end{bmatrix},\ 
    \mathbf{M} = \bm{\mathcal{M}} - \tau \mathbf{A}_H,\  \mathbf{N} = -\tau \mathbf{C},\ \mathbf{A} = -\tau \mathbf{B}, 
\]
where $\tau$ is the time step size. We 
compare the performance of TriCG and TriCG-DR+D-TriCG for solving $10$
successive SQD linear systems. We set the parameters
\[
    \tau = 0.1,\ \mathtt{tol} = 10^{-10},\ \varepsilon_\mathrm{svd}= 10^{-10},\ k = 100,\ p = 200,\ \mathtt{maxcycle} = 10,\ \mathtt{maxit} = 2000.
\]
The first system is solved by TriCG-DR($200,100$). For this numerical example, the approximate elliptic singular
vectors corresponding to the $k=100$ largest elliptic singular values converge at the end of the $3$rd cycle with the error $4.26 \times
10^{-15}$. The D-TriCG method are employed for the subsequent SQD linear systems.
\Cref{fig:fig3_4} displays the convergence histories of TriCG and TriCG-DR+D-TriCG,
while \Cref{tab:tab2} presents the corresponding computational runtime. The proposed TriCG-DR+D-TriCG method achieves significant reductions in both iteration count and wall-clock time by leveraging spectral information from the approximate elliptic singular vectors.

\begin{figure}[htbp]
    \centering
    \includegraphics[width=0.48\linewidth]{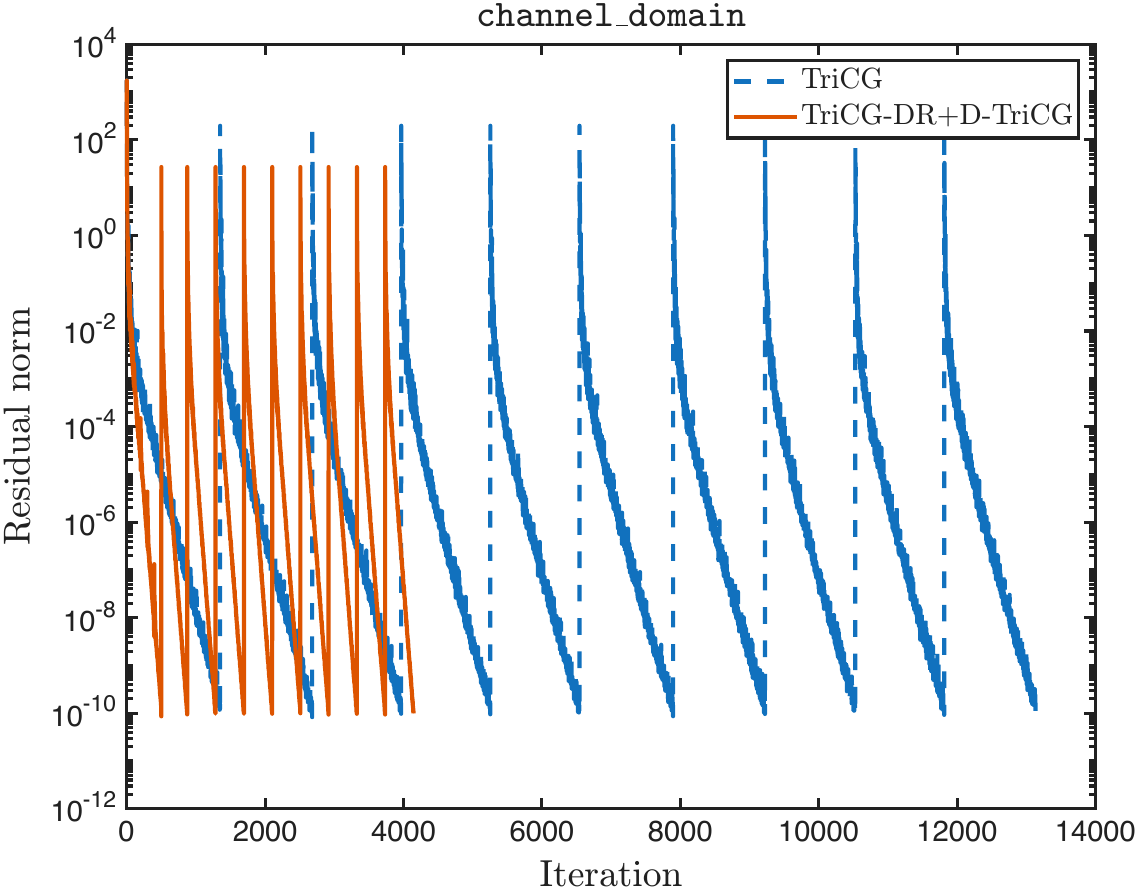}
    \caption{The convergence histories of TriCG and TriCG-DR+D-TriCG for
    10 right-hand sides on the problem \texttt{channel\_domain}.}
    \label{fig:fig3_4}
\end{figure}

\begin{table}[htbp]
    \centering
    \caption{CPU time of TriCG and TriCG-DR+D-TriCG on the problem \texttt{channel\_domain}.}
    \label{tab:tab2}
    \begin{tabular}{@{}rrr@{}}
        \toprule
        & TriCG & TriCG-DR+D-TriCG \\ \midrule
        Time(s) & $163.42$ & $108.85$ \\ \bottomrule
    \end{tabular}
\end{table}

\section{Concluding remarks and future work} 
\label{sec:conclusion}
When the off-diagonal block of the SQD matrix contains a substantial number of large elliptical singular values, TriCG exhibits relatively slow convergence. To address this issue, deflation techniques aimed at mitigating the impact of these large elliptical singular values can be utilized to accelerate the convergence of TriCG. Given the exact elliptic singular value decomposition (ESVD) of matrix $\mathbf{A}$, we demonstrate that the deflated system \eqref{eq:deflated2} can be solved via TriCG by merely modifying the right-hand side. However, in practical computational scenarios, the exact ESVD is usually not available. To address this limitation, we proposed the gSSY-DR method for computing several approximate elliptic singular triplets. Combining TriCG and gSSY-DR, we proposed TriCG-DR for solving SQD linear systems. Numerical experiments demonstrate that when the off-diagonal matrix $\mbf A$ contains a substantial number of large elliptic singular values, TriCG-DR achieves a significant reduction in iteration count and achieves marked acceleration in CPU runtime compared to TriCG. 

For SQD linear systems with multiple right-hand sides, the proposed D-TriCG method uses the approximate elliptic singular vectors that were computed by TriCG-DR while solving the first right-hand side system to generate an initial guess, then applies TriCG (some reorthogonalization steps are used to control the rounding errors) to compute the solutions of the systems with subsequent right-hand sides. Numerical experiments on the unsteady incompressible Stokes equation demonstrate significant convergence acceleration of the proposed TriCG-DR+D-TriCG method.

TriMR \cite{montoison2021tricg} is another method for solving SQD linear systems based on the minimal residual (MR) condition. Existing deflation techniques for MR-type Krylov subspace methods (see, e.g., \cite{baglama2013augmented,daas2021recycling,morgan2002gmres,abdelrehim2010deflated}) typically augment Krylov subspaces with harmonic Ritz vectors and treat the coefficient matrix as a whole. 
Our future research will focus on developing a novel deflation technique tailored for TriMR, which leverages the block two-by-two structure of \eqref{eq:problem} to enhance convergence.

\section*{Declarations}
\subsection*{Funding} 
This work was supported by the National Natural Science Foundation of China
(Nos. 12171403 and 11771364), and the Fujian Provincial Natural Science
Foundation of China (No. 2025J01031).
\subsection*{Conflict of Interest} 
The authors have no competing interests to declare that are relevant to the
content of this article.	
\subsection*{Data Availability} 
The MATLAB scripts to reproduce the results in this section are available at https://github.com/kuidu/tricgdr.
\subsection*{Author Contributions}
Both authors have contributed equally to the work.

%\small
%\bibliographystyle{abbrv}
%\bibliography{tricgdr}

\end{document}